\documentclass[10pt]{amsart}

\usepackage{amsmath, amsthm, amssymb, graphicx, enumerate, mathrsfs, everypage, mathtools, datetime, getfiledate, pifont} 

\usepackage{wrapfig}

\newcounter{citedtheorems}

\newcounter{theoremcounter}

\newtheorem{defn}[theoremcounter]{Definition}
\newtheorem{theorem}[theoremcounter]{Theorem}
\newtheorem*{theorem-m}{Theorem \ref{main-theorem}}
\newtheorem*{thm-m}{Main Theorem}
\newtheorem{thm-star}[theoremcounter]{Theorem} 

\newtheorem*{theorem-x}{Theorem}
\newtheorem*{theorem-abs1}{Theorem \ref{ind-theorem}}
\newtheorem*{theorem-abs2}{Theorem \ref{a23}}
\newtheorem*{theorem-abs3}{Theorem \ref{ind-new}}
\newtheorem*{theorem-abs4}{Theorem \ref{m1}}
\newtheorem{main-claim}[theoremcounter]{Main Claim}

\newtheorem{thm-lit}[citedtheorems]{Theorem}
\newtheorem{defn-lit}[citedtheorems]{Definition}
\newtheorem{fact-lit}[citedtheorems]{Fact}
\newtheorem{fact}[theoremcounter]{Fact}

\newtheorem{cor}[theoremcounter]{Corollary}
\newtheorem{defn-claim}[theoremcounter]{Definition/Claim}

\newtheorem{mlem}[theoremcounter]{Meta-lemma}
\newtheorem{summary}[theoremcounter]{Summary}

\newtheorem{concl}[theoremcounter]{Conclusion}
\newtheorem{conv}[theoremcounter]{Convention}
\newtheorem{claim}[theoremcounter]{Claim}

\newtheorem{subclaim}[theoremcounter]{Subclaim}

\newtheorem{lemma}[theoremcounter]{Lemma}
\newtheorem{obs}[theoremcounter]{Observation}
\newtheorem{rmk}[theoremcounter]{Remark}

\newtheorem{ntn}[theoremcounter]{Notation}
\newtheorem{disc}[theoremcounter]{Discussion}

\newtheorem{expl}[theoremcounter]{Example}

\newtheorem{qst}[theoremcounter]{Question}

\newtheorem{hyp}[theoremcounter]{Hypothesis}

\newcommand{\tfeq}{ T_{\operatorname{feq}} }

\newcommand{\br}{\vspace{3mm}}

\newcommand{\vsbr}{\vspace{1mm}}
\newcommand{\tcb}{ } 
\newcommand{\mkfin}{\mk_{\operatorname{fin}}}

\newcommand{\ml}{\mathcal{L}}
\newcommand{\tlf}{\trianglelefteq}

\newcommand{\cf}{\operatorname{cof}}

\newcommand{\dom}{\operatorname{Dom}}

\newcommand{\mcf}{\mathcal{F}}

\newcommand{\lgn}{\ell} 

\newcommand{\lex}{<_{\operatorname{lex}}}

\newcommand{\tpqf}{\operatorname{tp}_{\operatorname{qf}}}

\newcommand{\upk}{\Upsilon_{\mk}}
\newcommand{\qftp}{\operatorname{tp}_{\operatorname{qf}}}

\newcommand{\tp}{\operatorname{tp}}

\newcommand{\up}{\Upsilon}

\newcommand{\mct}{\mathcal{T}}

\newcommand{\ii}{\mathbf{i}}

\newcommand{\de}{\mathcal{D}}

\newcommand{\ts}{\mathbf{S}}

\newcommand{\trv}{\mathbf{t}} 

\newcommand{\rstr}{\upharpoonright}

\newcommand{\range}{\operatorname{range}}

\newcommand{\vp}{\varphi}

\newcommand{\ma}{\mathbf{a}}
\newcommand{\mb}{\mathbf{b}}

\newcommand{\trg}{T_{\mathbf{rg}}}
\newcommand{\tdlo}{T_{\mathbf{dlo}}}
\newcommand{\GEM}{\operatorname{GEM}}

\newcommand{{\xw}}{\mathbf{w}}
\newcommand{\xr}{\mathfrak{r}}

\newcommand{\inc}{\operatorname{inc}}
\newcommand{\mk}{\mathcal{K}}

\newcommand{\dl}{<_{\operatorname{dlo}}}
\newcommand{\lequ}{\leq_{\Upsilon}}

\title{A new look at interpretability and saturation}

\author{M. Malliaris and S. Shelah} 

\thanks{\emph{Thanks:}   
Malliaris was partially supported by NSF 1553653, and at IAS by NSF 1128155 and a Minerva research foundation membership. 
Shelah was partially supported by European Research Council grant 338821.  Both authors thank 
NSF grant 1362974 (Rutgers) and ERC 338821. 
This is paper 1124 in Shelah's list.}

\address{Department of Mathematics, University of Chicago, 5734 S. University Avenue, Chicago, IL 60637, USA}
\email{mem@math.uchicago.edu}

\address{Einstein Institute of Mathematics, Edmond J. Safra Campus, Givat Ram, The Hebrew
University of Jerusalem, Jerusalem, 91904, Israel, and Department of Mathematics,
Hill Center - Busch Campus, Rutgers, The State University of New Jersey, 110
Frelinghuysen Road, Piscataway, NJ 08854-8019 USA}
\email{shelah@math.huji.ac.il}
\urladdr{http://shelah.logic.at}

\begin{document}

\begin{abstract} 
We investigate the interpretability ordering $\tlf^*$ using generalized Ehrenfeucht-Mostowski models. This
gives a new approach to proving inequalities and investigating the structure of types. 
\end{abstract}

\maketitle

$T_0 \tlf^*_\kappa T_1$ in the interpretability order if, for sufficiently large regular $\lambda$, there is some $T_*$ which interprets 
both theories and which has the property that for any $\kappa$-saturated model $M_* \models T_*$, if the reduct of $M_*$ to $\tau(T_1)$ is 
$\lambda$-saturated, then so is the reduct to $\tau(T_0)$. 
It was introduced in the mid-90s as a potential 
help to the study of Keisler's order $\tlf$, which is defined via saturation of regular ultrapowers. 

Encouraged by our recent characterization of the maximal class in $\tlf^*$ under instances of GCH \cite{MiSh:1051} building on 
\cite{DzSh:692} and \cite{ShUs:844}, here we look further at $\tlf^*$. We prove a series of results about its structure,  
focusing on results which may give us insight into the structure of unstable theories, especially simple unstable theories.  
We prove  
the theory of the random graph is minimum among the unstable theories,
 and prove $\tfeq$ is minimum among the non-simple theories. 
We prove directly, i.e. without appealing to Keisler's order, that the theory of the random graph is not maximal. 
Finally, we prove directly that for any simple theory $T_0$ and any non-simple theory $T_1$, $T_1$ is not below $T_0$. To quote 
Keisler's order for this result would require assuming existence of an uncountable supercompact cardinal, so here both the proof and the 
theorem are new.  \tcb{(As indicated, $\tlf^*$ is often given with cardinal subscripts: as we'll explain in \S\ref{s:triangle-star}, 
here our main focus is $\kappa = 1$.) 

The proofs of the two separation results depend on sharpening the tool of Ehrenfeucht-Mostowski models so as to allow for a 
certain relative measurement of types. We plan to study this further in a companion manuscript in preparation. 

This paper has benefitted from very helpful comments of N. Ramsey, the members of the Berkeley model theory seminar, F. Parente, W. Boney,   
and the anonymous referee.

\setcounter{tocdepth}{1}
\tableofcontents

\section{What is the interpretability order $\tlf^*$?} \label{s:triangle-star}
\setcounter{theoremcounter}{0}

The interpretability order $\tlf^*$ was introduced in Shelah 1996 \cite{Sh:500} as a weakening of Keisler's order $\tlf$ \cite{keisler}. 
It was then studied in several subsequent papers, notably D\v{z}amonja and Shelah 2004 \cite{DzSh:692} and Shelah and Usvyatsov 2008 \cite{ShUs:844}.   In this section we will define $\tlf^*$,  
following \cite{DzSh:692}, and in the next section we will record what was known. All theories are complete. 

\begin{defn} \emph{(Interpretations, c.f. \cite{DzSh:692} 1.1)} \label{star1} 
Let $T_0$ and $T_*$ be complete first-order theories. Suppose that
\[ \overline{\vp} = \langle \vp_R(\overline{x}_R) : ~\mbox{$R$ a predicate or function symbol of $\tau(T_0)$, or ~~$=$}~~ \rangle \]
is such that each $\vp_R(\overline{x}_r) \in \tau(T_*)$. 

\noindent
(a) For any model $M_* \models T_*$, we define the model 
$N = {M_*}^{[\overline{\vp}]}$ as follows:
\begin{itemize}
\item $N$ is a $\tau(T_0)$-structure
\item $\dom(N) = \{ a : M_* \models \vp_{=}(a,a) \} \subseteq M_*$
\item for each predicate symbol $R$ of $\tau(T_0)$, $R^N = \{ \overline{a} : M_* \models \vp_R[\overline{a}] \}$
\item for each function symbol $f$ of $\tau(T_0)$ and each $b \in N$, 
$N \models$ ``$f(\overline{a}) = b$'' iff $M_* \models \vp_f(\overline{a}, b)$, and $M_* \models$ ``$\vp_f(\overline{a}, b) \land 
\vp_f(\overline{a}, c) \implies b = c$''. \footnote{This clause allows us to restrict to vocabularies with only relation symbols.}
\end{itemize}
(b) We call $\overline{\vp}$ an \emph{interpretation} of $T_0$ in $T_*$ if:
\begin{itemize}
\item each $\vp_R(\overline{x}_r) \in \tau(T_*)$ 
\item for any model $M_* \models T_*$, we have that ${M_*}^{[\overline{\vp}]} \models T_0$
\end{itemize}
(c) ``$T_*$ interprets $T_0$'' means: there exists $\overline{\vp}$ which is an interpretation of $T_0$ in $T_*$.  
\end{defn}

\tcb{In the definition of $\tlf^*$, note there are naturally three parameters: the amount of saturation to be transferred,
the size of the interpreting theory $T_*$, and a base level of saturation required for models of $T_*$ before we 
ask about transfer of saturation.  These are denoted by $\lambda, \chi, \kappa$ respectively. In the present paper, we focus
on the cardinal $\kappa$, so (following a suggestion of the referee) we have given the main definition \ref{d:trs} with only 
this one cardinal, for clarity.} In the present paper, the theories will be countable, but this isn't necessary. 

\begin{ntn}
In Definition $\ref{d:trs}$ $\kappa$ is $1$ or an infinite cardinal. In that context, ``for every $1$-saturated model'' means simply: for every model. \footnote{F. Parente has suggested using subscript ``$0$'' instead of ``$1$'' for full symmetry of notation.}  
\end{ntn}

\begin{defn} \emph{(The interpretability order $\tlf^*$, c.f \cite{DzSh:692} 1.2 and \cite{Sh:500} 2.10)} \label{d:trs} 

\noindent Let $T_0$ and $T_1$ be complete first-order theories, and let $\kappa$ be $1$ or an infinite cardinal. 
\[ T_0 \tlf^*_\kappa T_1 \]
means that for all large enough regular $\lambda$, there exists $T_*$ of cardinality $\leq |T_0| + |T_1|$ such that
\begin{enumerate}
\item $T_*$ interprets $T_0$ via $\overline{\vp}_0$, and $T_*$ interprets $T_1$ via $\overline{\vp}_1$, and 
without loss of generality the signatures of the two interpretations are disjoint.  
\item For every $\kappa$-saturated model $M_* \models T_*$, \emph{if} ${M_*}^{[\overline{\vp_1}]}$ is $\lambda$-saturated, then 
${M_*}^{[\overline{\vp_0}]}$ is $\lambda$-saturated. 
\end{enumerate}

\end{defn}

\begin{disc} \label{3-cardinals}
For compatibility with earlier papers and occasional full generality here, we include the original definition: 
 $T_0 \tlf^*_{\lambda, \chi, \kappa} T_1$  when there exists   
$T_*$ of cardinality $\leq |T_0| + |T_1| + \chi$ such that (a) $T_*$ interprets $T_0$ via $\overline{\vp}_0$, and $T_*$ interprets $T_1$ via $\overline{\vp}_1$, and without loss of generality the signatures of the two interpretations are disjoint, and (b) for every $\kappa$-saturated model $M_* \models T_*$, \emph{if} ${M_*}^{[\overline{\vp_1}]}$ is $\lambda$-saturated, then 
${M_*}^{[\overline{\vp_0}]}$ is $\lambda$-saturated.  Note that from the definition we have immediately that 
if $T_0 \tlf^*_{\lambda, \chi, \kappa} T_1$ and $\chi^\prime \geq \chi$, $\kappa^\prime \geq \kappa$ then $T_0 \tlf^*_{\lambda, \chi^\prime, \kappa^\prime} T_1$. 
\end{disc}

{Regarding $\kappa$, we will focus here on countable theories, and our investigations here show that the two cases $\kappa = 1$ and $\kappa = \aleph_1$ are 
interesting for different reasons; these might be called the superstable and stable case, respectively. The case $\kappa = \aleph_1$ 
retains a stronger analogy to regular ultrapowers, whereas the case $\kappa = 1$ allows for the introduction of powerful 
techniques from EM models, and will be our focus here. However, future investigations may illuminate other aspects, and so 
to allow for easy quotation, we have written $\kappa$ throughout the paper.}  

\begin{obs}
If $T_0 \tlf^*_{\kappa} T_1$ and $\kappa^\prime \geq \kappa$ then $T_0 \tlf^*_{\kappa^\prime} T_1$.
\end{obs}

\begin{cor} \label{corp3}
\tcb{In particular, if $T_0 \tlf^*_1 T_1$, then $T_0 \tlf^*_{\aleph_1} T_1$, and if $\neg (T_0 \tlf^*_{\aleph_1} T_1)$, then 
$\neg (T_0 \tlf^*_1 T_1)$.} 
\end{cor}

\noindent For easy reference, we include the following immediate translation of Definition \ref{d:trs}. 

\begin{summary} \label{s:sum} \emph{ }

\begin{enumerate}
\item \emph{To show $T_0 \tlf^*_\kappa T_1$  means to show that for all large enough regular $\lambda$, there exists 
$T_*$ of size $\leq |T_0|+ |T_1|$ 
interpreting $T_0$ via some $\bar{\vp}_0$ and $T_1$ via some $\bar{\vp}_1$,   
such that for every $M_* \models T_*$ which is $\kappa$-saturated, {if} ${M_*}^{[\overline{\vp_1}]}$ is $\lambda$-saturated then 
${M_*}^{[\overline{\vp_0}]}$ is $\lambda$-saturated.}

\item \emph{To show $\neg ( T_0 \tlf^*_\kappa T_1 )$ means to show that for arbitrarily large regular $\lambda$, for 
every $T_*$ $($of size no more than $|T_0| + |T_1|$$)$ 
interpreting $T_0$ via some $\bar{\vp}_0$ 
and $T_1$ via some $\bar{\vp}_1$, there exists some $\kappa$-saturated 
 $M_* \models T_*$ such that ${M_*}^{[\overline{\vp_1}]}$ is $\lambda$-saturated but ${M_*}^{[\overline{\vp_0}]}$ is not $\lambda$-saturated.} 
$($And clearly it suffices to show that for arbitrarily large regular $\lambda$, for every $T_*$ interpreting our two theories, there exists some extension $T_{**} \supseteq T_*$ of the 
same cardinality, e.g. with 
Skolem functions, and some $\kappa$-saturated 
 $M_* \models T_{**}$ such that ${M_*}^{[\overline{\vp_1}]}$ is $\lambda$-saturated but ${M_*}^{[\overline{\vp_0}]}$ not $\lambda$-saturated.$)$ 

\item \emph{$T_0$ and $T_1$ are $\tlf^*_\kappa$-equivalent when $T_0 \tlf^*_\kappa T_1$ and $T_1 \tlf^*_\kappa T_0$, and they are $\tlf^*_\kappa$-incomparable when 
$\neg ( T_0 \tlf^*_\kappa T_1 )$ and $\neg ( T_1 \tlf^*_\kappa T_0 )$.} 

\item \emph{$T_0 \triangleleft^* T_1$ $($i.e. strictly less than$)$ when $T_0 \tlf^* T_1$ and $\neg ( T_0 \tlf^* T_1 )$. } 

\end{enumerate}

\end{summary}
Though in many cases this hypothesis won't be necessary, our focus here will be complete countable theories, because of the connection to Keisler's order, which is stated for such theories, and makes most sense for them, see \cite{Sh:14}.

\begin{conv} \label{c:ctble}
Unless otherwise stated, in this paper all theories are complete and countable, 
so also any relevant $T_*$ is countable. 
\end{conv}

\br

\section{What was known about $\tlf^*$?}  \label{s:known}

In this section we describe the state of knowledge on $\tlf^*_1$ and $\tlf^*_{\aleph_1}$  
as work on this paper began. Not all these results were previously known, 
e.g. as recently as \cite{MiSh:1051} we didn't record the structure on the stable theories, Theorem \ref{tr:st}, or that $\tlf^*_{1}$ strictly 
refines $\tlf$, Conclusion \ref{c:refines}. 
It may be most correct to say they could have been known: the results proved in this section may be deduced with a little thought from 
results in the literature.  

\tcb{The interpretability order $\tlf^*$ refines Keisler's order $\tlf$ in a natural sense as we now explain, but because the quantification over $\lambda$ in the 
two orders is different (all sufficiently large vs. all), we will keep track of $\lambda$ in the next few claims.} 
Recall that 
Keisler's order $\tlf$ on complete countable theories sets $T_0 \tlf T_1$ if $T_0 \tlf_\lambda T_1$ for every infinite cardinal $\lambda$, 
where $T_0 \tlf_\lambda T_1$ means that for every regular 
ultrafilter $\de$ on $\lambda$, every model $M_1 \models T_0$, and every model $M_2 \models T_1$, if $(M_2)^\lambda/\de$ is $\lambda^+$-saturated, then 
$(M_1)^\lambda/\de$ is $\lambda^+$-saturated.\footnote{A discussion of this order may be found in e.g. \cite{MiSh:1030} \S 2.} 
The subscripts $i,j$ in Claim \ref{c15} are for easier quotation. 

\begin{claim} \label{c15}  
\tcb{If $\neg (T_j \tlf_\lambda T_i)$ in Keisler's order, $\neg (T_j \tlf^*_{\lambda^+, \aleph_1} T_i)$. Thus by monotonicity, $\neg (T_j \tlf^*_{\lambda^+, 1} T_i)$.} 
\end{claim}

\begin{proof} 
By monotonicity, it will suffice to prove this for $\kappa = \aleph_1$. 
Suppose there were a regular ultrafilter $\de$ 
on $\lambda$ so that for any $M_\ell \models T_\ell$ for $\ell = i,j$,  
$(M_i)^\lambda/\de$ is $\lambda^+$-saturated but 
$(M_j)^\lambda/\de$ is not $\lambda^+$-saturated (recall that in Keisler's order, the choice of the models $M_\ell$ matters only 
up to elementary equivalence). 
Suppose for a contradiction that there were 
$T_*$ interpreting both $T_j$ and $T_i$ such that in any model of $T_*$ which is $\aleph_1$-saturated, if the reduct to 
$\tau(T_i)$ is $\lambda^+$-saturated then so is the reduct to $\tau(T_j)$. 
Let $M_* \models T_*$ and let $N_* = {M_*}^\lambda/\de$. Because $N_*$ is a regular ultrapower, it will be $\aleph_1$-saturated. 
As ultrapowers commute with reducts, 
$N_* \rstr_{\tau(T_i)}$ will be $\lambda^+$-saturated but 
$N_* \rstr_{\tau(T_j)}$ will not be $\lambda^+$-saturated. This contradicts the assumption on $T_*$ 
and shows no such $T_*$ exists, i.e. $\neg (T_j \tlf^*_{\lambda^+,\aleph_1} T_i)$. 
\end{proof}

\begin{cor} \label{triangle-refines} 
\tcb{Let $\kappa \in \{ 1, \aleph_1\}$. 
If $T_0 \tlf^*_\kappa T_1$ 
then $T_0 \tlf_\lambda T_1$ for all sufficiently large $\lambda$.} 
\end{cor}

\begin{proof} 
Recall that $T_0 \tlf^*_\kappa T_1$ means that for all large enough regular $\rho$, $T_0 \tlf^*_{\rho, \kappa} T_1$.  
It follows that for all sufficiently large $\lambda$, $T_0 \tlf^*_{\lambda^+, \kappa} T_1$, so 
apply the contrapositive of Claim \ref{c15} with $j=0$, $i=1$ to conclude $T_0 \tlf_\lambda T_1$. 
 \end{proof} 

\begin{cor} \label{c16} 
\tcb{Let $\kappa \in \{ 1, \aleph_1\}$. 
If for arbitrarily large 
$\lambda$ we have $T_0 \triangleleft_\lambda T_1$ in Keisler's order, then $\neg (T_1 \tlf^*_\kappa T_0)$.} 
\end{cor}

\begin{proof}
If $T_0~ \triangleleft_\lambda ~T_1$, the strict inequality means that $\neg (T_1 \tlf_\lambda T_0)$, so apply Claim \ref{c15} with $j=1$, $i=0$. 
\end{proof}

\begin{concl}[$\tlf^*$ refines $\tlf$ in a natural sense]  \label{cor-finer}
\tcb{If $T_0$ and $T_1$ are equivalent in $\tlf^*_{1}$ or in $\tlf^*_{\aleph_1}$ 
then they are equivalent in $\tlf_\lambda$ 
for all sufficiently large $\lambda$.}  
\end{concl} 

To show the equivalence relation induced by $\tlf^*$ is strictly finer than $\tlf$, we will need a few facts. 

\begin{disc} At least a priori, 
Corollary \ref{cor-finer} does not imply that the ordering on the Keisler classes must be inherited by the $\tlf^*$-classes. A priori,  
all or many of the $\tlf^*$-classes could be pairwise incomparable. 
\end{disc}

First recall that if $T_* \supseteq T$, the class $PC(T_*, T)$ is the class of reducts to $\tau(T)$ of models of $T_*$. 

\begin{fact}[\cite{Sh:c}, Theorem VI.5.3 p. 383] \label{ct-fact}
If $T$ is countable, superstable, and does not have the f.c.p., then there is 
$T_*$, $T \subseteq T_*$, $|T| = 2^{\aleph_0}$ such that $PC(T_*, T)$ is 
categorical in every cardinality $\geq 2^{\aleph_0}$.  
[The proof goes by showing one can choose $T_*$ so that the reduct to $T$ of any model of $T_*$ of 
cardinality at least continuum is saturated in its own cardinality]. 
\end{fact}

As an immediate corollary, proved just as in the proof of \ref{c:refines}, we have the following. 
(Recall from \ref{3-cardinals} the notation means $|T_*| \leq |T_0| + |T_1| + 2^{\aleph_0}$, and $\kappa = 1$.)
 
\begin{cor}
If $T$ is countable, superstable, and does not have the f.c.p., $T$ is $\tlf^*_{\lambda, 2^{\aleph_0}, 1}$-minimal among 
complete countable theories for every $\lambda \geq 2^{\aleph_0}$.  
\end{cor}

Second, following \cite{BlGrSh:570}, and note we now return to our default value $\chi = \aleph_0$, 

\begin{defn}
We say that saturation is $(\mu, \kappa)$ transferrable in $T$ (the interesting case is $\mu < \kappa$) 
if there is an expansion $T_* \supseteq T$ 
with $|T_*| = |T|$ such that if $M$ is a $\mu$-saturated model of $T_*$ and $|M| \geq \kappa$ then the 
reduct of $M$ to $\tau(T)$ is $\kappa$-saturated. 
\end{defn} 

Using this notion of transfer of saturation, Baldwin, Grossberg, and 
Shelah characterized four classes of countable theories, one of which we'll use here. 

\begin{fact}[\cite{BlGrSh:570}, p. 11] \label{bgs-fact} {Let $\lambda$ be an uncountable cardinal and $T$ a countable theory.}
Then 
($T$ is superstable and does not have the finite cover property) $\iff$ (saturation is $(\aleph_0, \lambda)$-transferrable for $T$).
\end{fact}

\begin{concl} \label{c:refines} 
Let $T_0, T_1$ be complete countable theories. 
If $T_0$ is superstable and does not have the finite cover property, and $T_1$ is strictly stable and does not have the finite cover property, 
then $T_0$ and $T_1$ are equivalent in Keisler's order $\tlf$ but not equivalent in $\tlf^*_{\lambda, {\aleph_0}, \aleph_0}$, 
\tcb{\emph{thus} not equivalent in $\tlf^*_1$}.

\emph{Thus}, the equivalence relation on countable theories induced by $\tlf^*_{1}$ is strictly finer than that induced by Keisler's order,  
already on the stable theories.
\end{concl}

\begin{proof}
We will prove this for $\kappa = \aleph_0$, which suffices for $\kappa = 1$ by monotonicity. 
The class of complete countable theories without the f.c.p. form an equivalence class in Keisler's order, see \cite{Sh:c} VI.5.1, so it will suffice 
to show that for $T_0, T_1$ as in the statement, $\neg (T_1 \tlf^*_{1} T_0)$. 
Suppose for a contradiction there were some theory $T_*$ interpreting both $T_0$ and 
$T_1$ witnessing that $T_1 \tlf^*_{1} T_0$ for all regular $\lambda$ 
above some $\lambda_*$.  Choose a regular $\lambda > \lambda_*$.  
There is no harm in expanding $T_*$ to $T_{**}$ witnessing that 
saturation is $(\aleph_0, \lambda)$-transferrable for $T_0$, i.e. if $M$ is an $\aleph_0$-saturated model of $T_{**}$ 
of cardinality $\geq \lambda$ then the reduct of $M$ to $\tau(T_0)$ is $\lambda$-saturated. Now, by the characterization of Fact \ref{bgs-fact}, 
there can be no theory witnessing such a transfer of saturation for $T_1$; in particular, $T_{**}$ cannot be such a theory, 
and so there must be a counterexample, namely a model $N \models T_{**}$ of cardinality $\geq \lambda$ such that $N$ 
is $\aleph_0$-saturated but the reduct of $N$ to $\tau(T_1)$ is not $\lambda$-saturated. However, since $N \models T_{**}$, its 
reduct to $\tau(T_0)$ is $\lambda$-saturated.  Since $N$ is a model of $T_*$, this contradiction completes the proof. 
\end{proof}

In the companion paper \cite{MiSh:F1692} being written we sort out the analogous cases, and will conclude there that 
$\tlf^*_1$  
has six classes on the countable stable theories, including incomparable classes. 

Returning to the case of $\tlf^*_{\aleph_1}$, on the stable theories the picture is the same 
as that given by Keisler's order.   The ordering $\tlf^*_{\aleph_1}$ does have incomparable classes on the unstable theories, however;
see below. 

\begin{theorem} \label{tr:st}
On the complete, countable, stable theories, $\tlf^*_{\aleph_1}$ has precisely two classes, those theories without the finite cover property 
and those theories which are stable but have the finite cover property. 
\end{theorem}

\begin{proof}
We adapt the proof that Keisler's order on the stable theories has precisely two classes, 
\cite{Sh:c} VI.5.1 (note that regular ultrapowers of countable theories are always $\aleph_1$-saturated).  
As there, the key background theorem says that a model of a countable stable theory is $\lambda$-saturated iff it is 
$\aleph_1$-saturated and every maximal indiscernible set has size at least $\lambda$. 
For one direction, by \cite{BlGrSh:570} Theorem 2.2, a complete countable stable $T$ has nfcp iff 
saturation is $(\aleph_1, \lambda)$-transferrable for all uncountable $\lambda$. Proceeding just as in the proof of   
\ref{c:refines}, this shows that the theories without the fcp are precisely the 
$\tlf^*_{\lambda, \aleph_0, \aleph_1}$-minimum class. 

For the other side 
it suffices to show that if $T_1, T_2$ are both stable with the fcp then 
$T_1 \tlf^*_{\aleph_1} T_2$.  Here 
we may follow \cite{Sh:c} VI.5.1(2)  p. 379, which proves that for 
$M$ a model of a countable stable theory $T$ with fcp, any regular ultrapower $M^I/\de$ is $\lambda$-saturated, where 
$\lambda$ is the minimum product of an unbounded sequence of finite 
cardinals mod $\de$. In fact, the proof there also shows that the ultrapower will not be $\lambda^+$-saturated; 
this is true here too, but it's simpler to just get two classes by quoting \ref{c16}. 
The argument will go through essentially verbatim except for one point: since we 
are not in a regular ultrapower, we'll need to justify the following: there is $T_*$ interpreting both 
$T_1$ and $T_2$ such that if $N$ is any model of $T_*$ whose reduct to $\tau(T_2)$ is $\lambda$-saturated, then in 
$N \rstr \tau(T_1)$, every pseudofinite set has size at least $\lambda$. (What does ``pseudofinite'' mean? 
We require that $T_*$ expands the theory of $(\mathbb{N}, <)$, and so here we are asking that any infinite,  
bounded subset of $\mathbb{N}^N$ which is definable with parameters in $T_*$, has size at least $\lambda$.)

Let's now justify this point. 
Recall that in a stable theory, having the fcp is equivalent to 
having (perhaps in an imaginary sort) a definable 
equivalence relation with a class of size $n$ for each $n$. 
Call this equivalence relation $E$.  Let 
$M$ be a countable model expanding $(\mathbb{N}, <)$. Expand $M$ 
so that its theory interprets $T_1$ and $T_2$, without loss of generality in disjoint signatures. Suppose finally 
that $Th(M)$ codes enough set theory, 
or number theory, so that there is a parametrized family of functions 
$F : E \times E \rightarrow \mathbb{N}$ 
witnessing that for each finite $n$ and each definable 
subset $X$ of $M$ of size $\geq n$ there is a definable injection of 
the $E$-class of size $n$ into $X$. Let $T_* = Th(M)$. 
Suppose now that $N \models T_*$, that $N \rstr \tau(T_2)$ is $\lambda$-saturated, and let $\vp(x, \bar{c})$ be a bounded, definable 
subset of $\mathbb{N}^N$. 
Then by overspill, for some $a$ in some infinite $E^N$-equivalence class, 
$F^N(a,-)$ will map $E^N(a,-)$ injectively into $\vp[N, \bar{c}]$. By our saturation hypothesis, 
$E^N(a,-)$ has size at least $\lambda$, so $\vp[N, \bar{c}]$ does too, as desired.  
\end{proof} 

In the present paper we have focused on the unstable theories, with an eye towards simple theories and Keisler's order. 

The papers \cite{DzSh:692} and \cite{ShUs:844} investigated maximality 
under $\tlf^*$. Building on the first, the second established 
that under relevant instances\footnote{The proof depends on a theorem of \cite{DzSh:692} which assumes relevant instances of GCH.} 
of GCH, any theory with $NSOP_2$ is non-maximal in $\tlf^*$.  

\begin{fact}[\cite{ShUs:844} 3.15(2)] \label{844-fact}
Assuming relevant instances of GCH, $T$ is $SOP_2$ if $T$ is maximal in $\tlf^*_{\aleph_1}$. Note that 
$SOP_2$ is an absolute property of $T$, so that is, 
$T$ is $SOP_2$ if $T$ is maximal in some universe $($e.g. some forcing extension$)$ extending 
$\mathbb{V}$.
\end{fact}

Next, we quote a characterization of the maximal class under $\tlf^*$. The proof of the complementary direction to \ref{844-fact} given in 
\cite{MiSh:1051} is in ZFC and uses some ideas from the proof that $SOP_2$ is $\tlf$-maximum, from \cite{MiSh:998}.  Our proof there 
shows that $SOP_2$ suffices for $\tlf^*_1$-maximality, so a fortiori for $\tlf^*_{\aleph_1}$-maximality. 

\begin{fact}[\cite{MiSh:1051} Theorem 7.13] \label{maxl-fact} \tcb{Any theory with $SOP_2$ is $\tlf^*_1$-maximal.}
Thus, assuming relevant instances of GCH, $T$ is $\tlf^*_{\aleph_1}$-maximal if and only if it has $SOP_2$. 
\end{fact}

It would be nice to remove $GCH$ from \ref{844-fact} and therefore from the characterization of the maximal class, see \S \ref{s:op} Problem \ref{p:nsop2} below, but this seems to require new ideas: 
currently, GCH contributes to simplifying the structure of trees and thus to extracting a suitable 
amalgamation condition from the $NSOP_2$ hypothesis. 

The remainder of what (we believe) is known on $\tlf^*$ comes from quoting the known results on Keisler's order, in each case 
invoking \ref{c15} or \ref{c16} as appropriate, and 
noting that the known ZFC dividing lines in Keisler's order (including the infinitely many classes of \cite{MiSh:1050}) all hold for arbitrarily 
large $\lambda$.

\section{$\GEM$ models and indiscernible sequences} \label{s:gem}
\setcounter{theoremcounter}{0}

In this section we review the basic setup of generalized Ehrenfeucht-Mostowski models and establish that the classes of 
index models we'll use later in the paper have the desired properties. 
We roughly follow \cite{Sh:E59} \S 1,  but it seemed best 
to make the paper self contained. We prefer to call the models generalized EM models rather than just EM models 
to stress that we use different index models. 

Note to the reader: we've begun the paper with this introductory material to frame the approach, however, 
the reader interested only in the coding arguments of \S\ref{s:trg} and \S\ref{s:tfeq} may skip ahead.

\begin{defn} \label{d:generates}
We say that $\ma = \{ \bar{a}_t : t \in I \} \subseteq {^{\omega >}N}$ \emph{generates} the model $N$ if every element of $N$ is in the definable closure of $\{ \bar{a}_{t,\ell} : t\in I, \ell < \lgn(\bar{a}_t) \}$.  
\end{defn}

In the next definition, ``quantifier-free types'' means ``...of finite tuples.''

\begin{defn} 
A \emph{template} $\Phi$ is a map whose domain is the set of quantifier-free types of 
one structure and whose range is contained in the set of quantifier-free types of a possibly different structure.
\end{defn}

\begin{expl} \label{ex1}
The classical example of a template comes when $(I, <)$ is an infinite linear order, $N$ is a model, and 
$\langle \bar{a}_t : t \in I \rangle$ is an indiscernible sequence in $N$: the map $\Phi$ taking 
$\tpqf( \langle t_0, \dots, t_{k-1} \rangle, \emptyset, I )$ to $\tpqf ( \bar{a}_{t_0}~^\smallfrown \cdots ~^\smallfrown \bar{a}_{t_{k-1}}, 
\emptyset, N)$ is a template.  
\end{expl}

A related but much richer source of examples arise as follows. 

\begin{ntn} \label{bar-t} \emph{ }
When $\bar{t} = \langle t_{i_0}, \dots, t_{i_{k-1}} \rangle \in {^{k} I}$, write $\bar{a}_{\bar{t}}$
for ${\bar{a}_{t_{i_0}}} ~^\smallfrown \cdots ^\smallfrown \bar{a}_{t_{i_{k-1}}}$. 
\end{ntn}

\begin{defn}[$\GEM$ models and proper templates, \cite{Sh:E59} Definition 1.8] \label{d:t-e59} \emph{ }
We say 
$N = \GEM(I, \Phi) = \GEM(I, \Phi, \ma)$ is a generalized Ehrenfeucht-Mostowski model with skeleton ${\ma}$ when
for some vocabulary $\tau = \tau_\Phi$ we have: 
\begin{enumerate}
\item $I$ is a model, called the index model. 
\item $N$ is a $\tau_\Phi$-structure and ${\ma} = \{ \bar{a}_t : t \in I \}$ generates $N$. 
\item $\langle \bar{a}_t : t \in I \rangle$ is quantifier free indiscernible in $N$. 
\item $\Phi$ is a template, taking $($for each $n<\omega$$)$ the quantifier free type of 
$\bar{t} = \langle t_0, \dots, t_{n-1} \rangle$ in $I$ to the quantifier free type of 
$\bar{a}_{\bar{t}}$ in $N$. $($So $\Phi$ determines $\tau_\Phi$ uniquely, and also 
a theory $T_\Phi$, the maximal $\tau_\Phi$-theory which holds in every such $N$.$)$ 
\end{enumerate}
\end{defn}

The skeleton $\ma$ generating a given $\GEM$ model may not be unique, so we often display it. Note also that templates 
are simply possible instructions, which may not be `coherent' or give rise to a model.  Templates which do have a special name. 

\begin{defn}
The template $\Phi$ is called proper for $I$ if 
there is $M$ such that $M = \GEM(I, \Phi)$. We say $\Phi$ is proper for a class $\mk$ if $\Phi$ is proper for 
all $I \in \mk$. 
\end{defn}

Here are some helpful properties we will assume our templates have. 

\begin{rmk} \label{d:rmk} Let $\Phi$ be a template proper for $\mk$. 
When $T_\Phi$ has Skolem functions, 
\begin{enumerate}
\item $\Phi$ is also \emph{nice} $($by transitivity of $<$$)$, which implies:   
\begin{enumerate}
\item $T_\Phi$ is complete. 
\item for every $I \in \mk$, $\langle \bar{a}_{t} : t \in I \rangle$ is indiscernible, not just quantifier-free 
indiscernible, in $\GEM(I, \Phi)$. 
\item for every $J \subseteq I$ from $\mk$ we have $\GEM(J, \Phi) \preceq \GEM(I, \Phi)$. 
\end{enumerate} 
\item $\GEM(I, \Phi)$ is unique in the sense that it depends, up to isomorphism, 
on $\Phi$ and the isomorphism type of $I$. 
\end{enumerate}
\end{rmk}

To summarize in the usual terminology of EM models: 

\begin{conv} \label{c:nice} 
All templates we consider are assumed to be \emph{very nice}, meaning: they are non-trivial $($i.e. we may add in $\ref{d:t-e59}(1)$ the condition
that $\lgn(\bar{a}_t) \geq 1$ and $\langle \bar{a}_t : t \in I \rangle$ is without repetition$)$, and $T_\Phi$ is well defined and 
has Skolem functions, \emph{thus} $\ref{d:rmk}$ applies $($so $T_\Phi$ is complete, etc., as there$)$. 
\end{conv}

\begin{ntn} \emph{ }
We may write $\GEM_{\tau}(I,\Phi)$ to denote the reduct to $\tau$. 
\end{ntn}

Next we explain our conventions and requirements on the class $\mk$ of index models. 

\begin{conv} \label{conv1} 
In what follows $\mk$ will always denote a nonempty class of infinite models, called \emph{index models}, which are expansions of linear orders, to a vocabulary $\tau = \tau(\mk)$, so $<$ belongs to $\tau$.  In particular $\mk$ need not be elementary.  We will use $I, J$... for elements of $\mk$.  
\end{conv}

\begin{conv} \label{conv2}
As $\mk$ may not be elementary, the phrase ``$J$ in $\mk$ is $\aleph_0$-saturated'' will always abbreviate ``$J$ is $\aleph_0$-homogeneous and 
$\aleph_0$-universal for elements of $\mk$.''  We will always assume our $\mk$s contain $\aleph_0$-saturated elements. 
\end{conv}

In this setup a crucial property of a class $\mk$ is being \emph{Ramsey}.   
To motivate this property, consider again the example \ref{ex1} of an indiscernible sequence $\bar{a} = \langle \bar{a}_t : t \in I \rangle$ in 
a model $N$, and its associated template $\Phi$.  Suppose we were to expand $N$ to $N^+$ in some larger language. 
The sequence $\bar{a}$ might no longer be indiscernible in $N^+$, but we could find in some elementary extension $N_*$ of $N^+$  
an indiscernible sequence $\bar{b} = \langle \bar{b}_t : t \in I \rangle$ such that the template $\Psi$ associated to $\bar{b}$ 
is an extension of $\Phi$ in a natural sense: 
\[ \mbox{ if } \Phi( \tpqf( \bar{t}, \emptyset, I) ) = p 
\mbox{ and } \Psi( \tpqf( \bar{t}, \emptyset, I) ) = q 
\mbox{ then } p = q \rstr \tau(\Phi). \]  
The right analogue for $\GEM$ models is given by the Ramsey property, which both tells us certain templates exist, and ensures that these templates reflect 
the given base structure. This definition is somewhat less general than \cite{Sh:E59} Definition 1.15, but it suffices for our purposes here.

\begin{defn} \label{d:r1}
We say the class $\mk$ is Ramsey, or simplified Ramsey, when: 

\noindent given any 
\begin{enumerate}[a)]
\item $J \in \mk$ which is $\aleph_0$-saturated,  
\item model $M$, and 
\item sequence ${\mb} = \langle \bar{b}_t : t \in J \rangle$ of finite sequences from $M$ with the length of $\bar{b}_t$ 
determined by $\tpqf(t, \emptyset, J)$, 
\end{enumerate}
there exists a template $\Psi$ which is proper for $\mk$  such that: 
\begin{enumerate}[i)]
\item $\tau(M) \subseteq \tau(\Psi)$ 
\item $\Psi$ reflects ${\mb}$ in the following sense: 

\noindent for any $s_0, \dots, s_{n-1}$ from $J$, 
\\ any $\vp = \vp(x_0, \dots, x_{m-1}) \in \ml(\tau(M))$, 
\\ and any $\tau(M)$-terms $\sigma_\ell(\bar{y}_0, \dots, \bar{y}_{n-1})$ for $\ell = 0, \dots, m-1$, 

\begin{quotation}
\vsbr
\noindent \underline{if} 
$ M \models \vp[\sigma_0(\bar{b}_{t_0}, \dots, \bar{b}_{t_{n-1}}), \dots, \sigma_{m-1}(\bar{b}_{t_0}, \dots, \bar{b}_{t_{n-1}})] $
\\ for every $t_0, \dots, t_{n-1}$ realizing $\tpqf(s_0~^\smallfrown \cdots ^\smallfrown s_{n-1}, \emptyset, J)$ in $J$,

\vsbr 
\noindent \underline{then} $\GEM(J, \Psi) \models \vp[\sigma_0(\bar{a}_{s_0}, \dots, \bar{a}_{s_{n-1}}), \dots, \sigma_{m-1}(\bar{a}_{s_0}, \dots, \bar{a}_{s_{n-1}})]$. 
\end{quotation}
\end{enumerate}
\end{defn}

\begin{disc}  
The term ``Ramsey property'' for an index model class is justified by the fact (in our language) that it  
corresponds naturally to the set of finite substructures of elements of the class 
being a Ramsey class in the sense of  Ne\v{s}et\v{r}il-R\"{o}dl \cite{nr} and
of Kechris-Pestov-Todor\v{c}evi\'{c} \cite{kpt}. See Scow \cite{scow2} Theorem 4.31 (in a slightly different language). 
\end{disc}

\noindent 
We will generally use Definition \ref{d:r1} in the form of Corollary \ref{d:ramsey-exp}. 

\begin{conv}
Given a class $\mk$, we may write $\upk$ for the class of templates proper for $\mk$, and write 
$\up$ when $\mk$ is clear from context. 
\end{conv}

\begin{defn} \label{d:order}
Given a class of templates $\Upsilon$, let $\leq_{\Upsilon}$ be the natural partial order 
on $\Upsilon$, that is, $\Phi \leq_{\Upsilon} \Psi$ means that $\tau(\Phi) \subseteq \tau(\Psi)$ and 
$\GEM(I,\Phi) \subseteq \GEM(I, \Psi)$, $\GEM_{\tau(\Phi)}(I,\Phi) \preceq \GEM_{\tau(\Phi)}(I, \Psi)$.  
We may just use $\leq$ when $\Upsilon$ is clear from context. 
\end{defn}

\begin{cor} \label{d:ramsey-exp} If $\mk$ is Ramsey,  whenever we are given:
\begin{enumerate}[a)]
\item $J \in \mk$ is $\aleph_0$-saturated
\item $\Phi$ a template proper for $\mk$
\item $M = \GEM(J, \Phi)$ with skeleton $\ma$ 
\item $N^+$, an elementary extension or expansion of $M$, or both   
\end{enumerate}
then there is a template $\Psi$ proper for $\mk$ with $\tau(\Psi) \supseteq \tau(N^+)$ and $\Psi \geq \Phi$. 
Moreover, $\Psi$ reflects $\ma$ in the sense of $\ref{d:r1}$ i)-ii), with $\ma$, $N^+$ here replacing $\mb$, $M$ there. 
\end{cor}

\begin{proof} 
If it is not already the case, expand $N^+$ to have Skolem functions.  
Apply \ref{d:r1} with $J, N^+, \ma$ here for $J, M, \mb$ there. 
\end{proof}

The reader will notice that also in the proofs below, before applying \ref{d:ramsey-exp}, we expand the models in question to have 
Skolem functions, and this is to ensure \ref{c:nice} above. 
Note: we are assuming all templates are nice (\ref{c:nice}), so what we have called ``Ramsey'' 
is sometimes called ``nicely Ramsey'' elsewhere, e.g. \cite{Sh:E59} Definition 1.16.

In this language, Ehrenfeucht and Mostowski \cite{EM} proved that: 

\begin{fact}
Let $\mk$ be the class of linear orders. Then $\mk$ is Ramsey. 
\end{fact}

It likewise generally requires proof to show that other classes are Ramsey. 
(See also the Appendix and the papers cited there.) 
In this paper we will use the following two classes which were already known to be Ramsey. 

The first is the class of linear orders expanded by $\mu$ unary predicates which partition the domain. Note that 
asking that the predicates partition the domain means this is not an elementary class (recall \ref{conv2}).  

\begin{defn} \label{d8}
Define $\mk_\mu$ to be the class whose elements are  
\[ I = (I ; <, \{ Q_{\alpha} : \alpha < \mu\} ) \]
where $(I, <)$ is a linear order, and the unary predicates $\langle Q^I_\alpha : \alpha < \mu \rangle$ partition $I$. 
\end{defn}

\begin{fact} \label{kormur}
For any cardinal $\mu$, $\mk_\mu$ is Ramsey. 
\end{fact}

\begin{proof} \cite{Sh:E59} Theorem 1.18(5). Alternately, using the language of Ramsey classes, one can
consider $\mkfin$ and cite the 
theorem due to Ne\v{s}et\v{r}il and R\"{o}dl \cite{nr} and independently to 
Abramson and Harrington \cite{ah}, that for every relational signature $\tau$ the class of all $\tau \cup \{ \leq \}$-structures, 
where $\leq$ denotes a linear order, is a Ramsey class. 
A statement and proof of this theorem is in \cite{vt} \S 4.1, Theorem 4.1. 
\end{proof}

\begin{defn} \label{n5} \emph{ }
\begin{enumerate}
\item $\mk^{tr}_\kappa$ is the class of trees with $\kappa$ levels and lexicographic order which 
are normal, meaning 
that a member $\eta$ at a limit level is determined by $\{ \nu : \nu \tlf \eta \}$.  
$($So the tree has the function $\cap (\eta, \nu) = \min \{ \rho : \rho \tlf \nu, \rho \tlf \eta \}$.$)$
\item We call $I \in \mk$ standard when the $i$th level, $P^I_i$, of $I$ consists of sequences of length $i$ and $n \in P_i$, $j<i$, 
$\eta \rstr j \in P_j$ and $\eta \rstr j \tlf_I \eta$, so every $I \in \mk^{tr}$ is isomorphic to a standard one 
$($this is justified by the assumption of normality$)$. 
\end{enumerate}
\end{defn}

\begin{fact}
$\mk^{tr}_\kappa$ is Ramsey. 
\end{fact}

\begin{proof}
This is \cite{Sh:c} VII \S 3, see 3.7 p. 424 and see 2.4 of the Appendix there. 
\end{proof}

\begin{disc} \label{disc:ramsey}
This setup suggests that new Ramsey theorems in combinatorics may directly help model-theoretic classification insofar as  
new Ramsey theorems may allow for new comparisons of theories via $\GEM$-models.  
\end{disc}

\vspace{2mm}

\section{$\trg$ is minimal among the unstable theories}
\label{s:trg}
\setcounter{theoremcounter}{0}

\begin{ntn} \emph{ }
When $\vp$ is a formula, $\vp^0$ abbreviates $\neg \vp$ and $\vp^1$ abbreviates $\vp$. 
\end{ntn}

\begin{lemma}[Minimality of the random graph] \label{l:trg}  \emph{ }
$\trg$, the theory of the random graph, is $\tlf^*_1$-minimal among the complete, countable unstable theories.   
\emph{Thus}, by monotonicity, $\trg$ is $\tlf^*_{\aleph_1}$-minimal among such theories. 
\end{lemma}

The proof of Lemma \ref{l:trg} will show more, namely: 

\begin{cor}
For $T_0 = \trg$ and any 
unstable $T_1$, $T_0 \tlf^*_1 T_1$ is witnessed by a theory 
$T_*$ expanding $Th(\mathbb{N}, <)$ with the property that if $N \models T_*$ and $N \rstr \tau(T_1)$ is $\lambda$-saturated, then 
$\mathbb{N}^N$ has cofinality at least $\lambda$.  
\end{cor}

For the proof of Lemma \ref{l:trg}, we'll need a claim.  

\begin{claim} \label{c:indep}
Let $T$ be a complete countable theory and suppose $\vp(\bar{x};\bar{y})$ has the independence property for $T$. 
Then there is a countable model $M$ of $T$ and a sequence $\langle \bar{b}_n : n < \omega \rangle$ with $\ell(\bar{b}_n) = \ell(\bar{y})$ 
contained in $M$, 
over which $\vp$ has the independence property, and  
such that for any $\bar{a} \in {^{\ell(\bar{x})}M}$, there is a truth value $\trv \in \{ 0, 1 \}$ such that for all but finitely many $n$, 
$M \models \vp[\bar{a}, \bar{b}_n]^\trv$.
\end{claim}

\begin{proof}
For simplicity, we will write this proof as if $\bar{x}$, $\bar{y}$ were singletons, but the proof is identical for finite tuples.
Let $N_1$ be an $\aleph_1$-saturated model of $T$. Choose $\langle c_\alpha : \alpha < \omega \rangle$ 
such that in $N_1$ the set of formulas $\{ \vp(x,c_\alpha) : \alpha < \omega \}$ is independent. Let $M \preceq N_1$ be countable 
and contain $\{ c_\alpha : \alpha < \omega \}$. Let $\langle d_i : i < \omega \rangle$ list the elements of $M$. 
Let $\de$ be a uniform ultrafilter on $\omega$. For each $i$ let $\trv(d_i)$ be such that 
$C_i := \{ c_\alpha : M \models \vp[d_i, c_\alpha]^{\trv(d_i)} \} \in \de$. 

By induction on $n<\omega$ choose distinct elements $b_n$
such that $b_n \in \bigcap_{i<n} C_{i}$. 
The sequence of $b_n$'s is chosen as a subsequence of the $c_\alpha$'s, so $\vp$ will 
a fortiori have the independence property on this sequence. 
Moreover, if $a \in M$ then for some $i$, $a = d_i$. So 
once $n \in (i, \omega)$, $b_n \in \bigcap_{i < n} C_i$, thus  $M \models \vp[a, b_n]^{\trv(a)}$ as desired. 
\end{proof}

\begin{proof}[Proof of Lemma \ref{l:trg}]
We will prove that $\trg \tlf^*_1 T_1$ for $T_1$ any countable, unstable, complete first order theory. If $T_1$ has $SOP$, then it is already maximal, 
so it will suffice to prove this in the case that $T_1$ has the independence property. 

Let $\vp(\bar{x}, \bar{y})$ be a formula which has the independence property for $T_1$. (In what follows, we'll write as if $\ell(x) = \ell(y) =1$, 
but this is only for simplicity of notation.)   
We'll also assume the three theories $T_1$, $\trg$, and $Th(\mathbb{N}, <)$ have disjoint signatures.

Let $M$ be a countable model whose theory $T$ satisfies:
\begin{enumerate}[(a)]
\item $M$ expands $(\mathbb{N}, <)$.
\item $M \rstr \tau(\trg)$ is a countable random graph.  
\\ That is, for some unary predicate $P$ such that $P^{M}$ is countably infinite, and some binary relation $R$, 
$T \vdash$ ``$(P, R)$ is $n$-random for every $n$,'' i.e. for every finite $n$, for any two disjoint subsets $U, V \subseteq  P^{M}$
of size $n$, there exists $a \in P^{M}$ such that $R^{M}(a,u)$ for all $u \in U$ and $\neg R^{M}(a,v)$ for all $v \in V$.
\item $M \rstr \tau(T_1)$ is a model of $T_1$ satisfying the conclusion of Claim \ref{c:indep}.  The domain of this model is named by 
$Q^M$, and $S^M$ is a binary relation with $T \vdash S \subset Q \times Q$ and $T \vdash$ ``$S$ has the $n$-independence property 
for each $n$,'' that is, for each $n$, there exist $a_1, \dots, a_n$ in $Q^M$ such that the formulas ${\{ S(x,a_i) : 1 \leq i \leq n \}}$ 
are independent. (\emph{We just let $S$ name $\vp$; in other words, for simplicity, we forget everything about the model of $T_1$ except for its domain and the 
formula with the independence property.})
\item $F^{M}$ is an injective function from $\mathbb{N}$ into $Q^{M}$ such that for every $a \in Q^M$, 
 for some truth value $\trv$, for every $n$ large enough, $M \models \vp[a, F^{M}(n)]^{\trv}$. 
\\ (\emph{It suffices to let $F^M(n) = b_n$ where $\langle b_n : n < \omega \rangle$ is from Claim \ref{c:indep}.})
\item $G^{M}$ is a one to one and onto function from $\mathbb{N}$ onto $P^{M}$ (recall $P^M$ was defined in part (b)). 
\end{enumerate}

Note that by our construction, $T = Th(M)$ records: 
\begin{itemize}
\item[(i)] the range of $F$ is a set over which the formula $S$ has the independence property;

\item[(ii)] for every $u \subseteq \{ n : n < n_* \}$ which is definable\footnote{Here definable means in the language of $T$, possibly 
with parameters, i.e. not just in $\{ < \}$.}, for every finite $n_* \in \mathbb{N}$, 
[replacing $u$ by its definition] the set of formulas $\{ R(x,G(y)) : y \in u \} \cup \{ \neg R(x,G(y)) : y < n_*, y \notin u \}$ 
is a partial type and is realized in $P^M$.  In particular, replacing $u$ by its definition, in any model $N$ of $T$, this remains true for any 
$n_* \in \mathbb{N}^N$, possibly nonstandard. 
\end{itemize}
Now let $N$ be any model of $T$. It will suffice to prove that if $N \rstr \tau(T_1)$, or really just 
$N \rstr \{ Q, S \}$, is $\lambda$-saturated, then 
the following three facts hold. 

\br

\noindent (1) \emph{$\mathbb{N}^{N}$ has cofinality $\geq \lambda$.}

If not let $\langle a_\alpha : \alpha < \kappa \rangle$ be $<^N$-increasing and cofinal in $\mathbb{N}^N$, with $\kappa < \lambda$. 
Let $p(x)= \{ Q(x) \} \cup \{ S(x,F^{N}(a_{2\alpha})) \land \neg S(x,F^N(a_{2_\alpha+1}) : \alpha < \kappa \}$. 
Then $p$ is a finitely satisfiable type of cardinality $\kappa < \lambda$. But it can't be realized because, by item (d), 
every element of $Q$ has an eventual $S$-truth value with respect to the image of $F$.  (This was expressible in $T$, 
so remains true in $N$.) So 
$N \rstr \tau(T_1)$ is not saturated, contradiction. 

\br

\noindent (2) \emph{$(P^N, R^N)$ is $\lambda$-saturated.}

If $(P^N, R^N)$ is not $\lambda$-saturated, let $p(x)$ be a 1-type of cardinality $<\lambda$ there which is omitted. Without loss of generality 
$p(x) = \{ (R(x,a_\alpha)^{\mbox{if $\eta(\alpha)$}} : \alpha < \alpha_* \}$ for some $\eta \in {^{\alpha_*}2}$ and 
$\langle a_\alpha : \alpha < \alpha_* \rangle$ with no repetition. 
Let $b_\alpha = (G^{-1})^N(a_\alpha)$ for $\alpha < \alpha_*$. So $\{ b_\alpha : \alpha < \alpha_* \}$ is a subset of 
$\mathbb{N}^N$ of cardinality $<\lambda$ with no repetition, so it has an upper bound $b_*$ by (1). 
Let $c_\alpha = F^N(b_\alpha)$ for $\alpha < \alpha_*$, so $\langle c_\alpha : \alpha < \alpha_* \rangle$ is a sequence of 
$<\lambda$ elements with no repetition, so recalling our observation (i), $\{ S(x,c_\alpha)^{\mbox{if $\eta(\alpha)$}} : \alpha < \alpha_* \}$ is a type 
over $Q^N$. By assumption that $N \rstr \tau(T_1)$ is saturated, this type is realized, say by $d$. 
So the set $U = \{ b < b_{*} : S(d, F^N(b)) \}$ is a first-order definable subset of $N$ (with parameters) 
hence recalling our observation (ii), 
there is 
$a_{\alpha_*}$ in range $G$ such that $(\forall y < b_{*}) (R(a_{\alpha_*}, G(y)) \iff y \in u )$
i.e. $(\forall y < b_{*}) (~R(a_{\alpha_*}, G(y)) \iff S(d,G(y))~ )$
so $a_{\alpha_*}$ realizes $p$. 

\br

For item (3), recall from (1) that since $M$ expands $(\mathbb{N}, <)$, the model $N$ can also be thought of as a set of (possibly 
nonstandard) integers, call it $\mathbb{N}^N$. 
 In $M$, 
every subset of $\mathbb{N}$ which is bounded in the sense of the ordering is finite, so let us call any bounded subset of $\mathbb{N}^N$ 
\emph{pseudofinite}. This name is justified by the properties of such sets in $T$: every pseudofinite set, and a fortiori any nonempty 
definable subset of any pseudofinite set, is a discrete linear order with a first 
and last element.\footnote{This concept also played a key role in cofinality spectrum problems \cite{MiSh:998}.}  With these definitions, 
the following immediate property of $N$ is worth recording: 

\br

\noindent (3) \emph{$N$ satisfies $<\lambda$-regularity, meaning that every set of $<\lambda$ elements is contained in some pseudofinite set.}

This simply restates the fact that the model expands $\mathbb{N}$, and by (1) $\mathbb{N}^N$ has cofinality $\geq \lambda$, so 
every small subset of the model is contained in some bounded set of (nonstandard) integers. 
\end{proof}

\vspace{5mm}

\section{$\trg$ is not maximal} \label{s:trg-max}

Next we turn to a direct proof of the following theorem.

\begin{theorem} \label{t:trg-s}
$\trg$, the theory of the random graph, is not $\tlf^*_1$-maximal. 
\end{theorem}

\noindent Though to our knowledge not previously recorded, 
this follows from known results:

\begin{proof}[Indirect proof of Theorem \ref{t:trg-s}]  
First, we know that $\trg$ is minimum among the unstable theories in Keisler's order \cite{mm4}. 
\cite[Theorem 12.1]{MiSh:999} proves that for arbitrarily large $\lambda$, the random graph is {strictly} $\tlf_\lambda$-smaller 
than any non-low or non-simple theory in Keisler's order.  
By Observation $\ref{c16}$ above, 
we have that $\trg$ is not in the $\tlf^*_1$-maximum class. 
\end{proof}

Theorem \ref{t:trg-s} will also be a special case of Theorem \ref{t:simple} below, however, its proof is 
not a specialization of that proof, and already suggests the idea of \S \ref{s:wd}.

\begin{hyp} \label{h:mu}
For this section, unless otherwise stated, let $\lambda$, $\mu$ be arbitrary but fixed infinite cardinals satisfying 
$\lambda = \lambda^{<\mu}$ and $\lambda \geq 2^\mu$. 
\end{hyp}
  
Since the theory of linear order is $\tlf^*$-maximum, the following suffices: 

\begin{obs} \label{o:sect3}
To show $\trg$, the theory of the random graph, is not $\tlf^*_1$-maximum, it will suffice to show that for $\tdlo = Th(\mathbb{Q}, <)$, 
\[ \neg ( \tdlo \tlf^*_1 \trg ) \]
i.e. for arbitrarily large regular $\lambda$, for 
\emph{every} countable theory $T_*$ interpreting both $\trg$ and $\tdlo$ in the sense of Definition $\ref{d:trs}$, there is some model 
$M \models T_*$ such that $M \rstr \tau(\trg)$ is $\lambda$-saturated but $M \rstr \tau(\tdlo)$ is not $\mu$-saturated.  
Recalling $\ref{s:sum}(2)$, we may assume that $T_*$ has Skolem functions. 
\end{obs}

\begin{hyp} \label{n2b}
In the rest of this section we will assume: 
\begin{enumerate}
\item[(a)] $T_*$ is some theory containing Skolem functions, $|T_*| \leq \lambda$, such that: \footnote{The proof works for 
a larger class of $T_*$ than just the countable ones.}
\begin{enumerate}
\item[(i)] $T_*$ interprets $\trg$, with edge relation $R = R_{\operatorname{rg}}$, 
and 
\item[(ii)] $T_*$ interprets $\tdlo$, with ordering $< ~= ~<_{\operatorname{dlo}}$. 
\end{enumerate}  
\item[(b)] $\mk$ is the class of index models $\mk_\lambda$ from $\ref{d8}$.
 
\item[(c)] We will say ``$I \in \mk$ is separated'' when for all $\alpha < \lambda$, $|P^I_{\alpha}| \leq 1$.
 \end{enumerate}
\end{hyp}

Separated $I$ have exactly no more than one element of each color. Note that separation puts no restrictions on the linear order 
other than size, so e.g. there are separated $I \in K$ which contain $(\kappa_1, \kappa_2)$-cuts for any $\kappa_1 + \kappa_2 \leq \lambda$.

For the remainder of this section, will use the following class of templates. 

\begin{defn} \label{k4} 
Let $\Upsilon$ be the class of templates $\Phi$ such that:
\begin{enumerate}
\item[(a)] $ \Phi $ is proper for $\mk$, 
\item[(b)] $I \in \mk$ implies $\GEM (I, \Phi) \models T_*$, 
\item[(c)] $s < t$ implies $a_s <_{\operatorname{dlo}} a_t$, informally, ``the template represents the order.'' 
\end{enumerate}
We use $\leq = \leq_{\Upsilon}$ for the natural partial order on elements of $\Upsilon$, as in $\ref{d:order}$. 
\end{defn}

We'll start with the key step in the saturation half of the argument. 
{A comment on strategy:} in Claim \ref{m17} below we have a type over some $M = \GEM(I, \Phi)$, and we'd like to choose 
a larger template $\Psi$ which will always guarantee its realization.\footnote{Recall from \ref{d:order} that $\GEM_{\tau(\Phi)}(I,\Phi) \preceq \GEM_{\tau(\Phi)}(I, \Psi)$ when $\Phi \leq \Psi$.}
A naive approach might be to say: let's realize $p$ in some elementary extension $M^\prime$, name this realization by 
a constant $c$, and then apply \ref{d:ramsey-exp} with $I, \Phi, M, M^\prime$, in the hopes that the template $\Psi$ we get back 
will ensure that $\vp(c,\bar{a}_{\bar{t}})$ holds in $\GEM(I, \Psi)$ for every $\vp(x,\bar{a}_{\bar{t}})$ in the type. 
There are two concerns:
\begin{enumerate}
\item \ref{d:ramsey-exp}(a) requires the 
index model to be sufficiently saturated, and in \ref{m17} our $I$ is far from saturated. 
\item Templates smooth things out: if $c \in \tau(M^\prime)$ then $c \in \tau(\Psi)$, but 
just because $\vp(c,\bar{a}_{\bar{t}})$ holds in $M^\prime$ 
doesn't mean its truth will carry over to $\Psi$. Only things which occur uniformly (in the sense of 
\ref{d:r1}(c)(ii), with $M^\prime$ here for $M$ there) are sure to be picked up by $\Psi$. 
\end{enumerate}
As we will see, the way out of both these concerns is to first replace $I$ by a sufficiently saturated 
$J$, and replace $p$ by a corresponding larger type $q$. The type $q$ will be ``many copies of $p$'' in a suitable sense, 
and provided it is consistent, any new constant $c$ realizing it will be ``reflected'' to $\Psi$ as indeed being a realization.

\begin{claim} \label{m17} 
Assume $\Phi \in \up$, $I \in \mk$ is separated, $M = \GEM(I, \Phi)$ and $p \in \ts_{R}(M)$. Then there is $\Psi \in \up$ and a constant 
$c \in \tau(\Psi) \setminus \tau(\Phi)$ such that $\Phi \lequ \Psi$ and if $J \in \mk$ and $h$ embeds $I$ into $J$ then $c$ realizes $h(p)$, defined naturally. 
\end{claim}

\begin{proof}
 
We begin with $M = \GEM(I, \Phi)$ and $p \in \ts_R(M)$. Fix $J$ such that $I \subseteq J \in \mk$ and $J$ is $\aleph_0$-saturated. 
Let $N  = \GEM(J, \Phi)$. By \ref{c:nice} $M \preceq N$, so we will identify the sequence 
$\langle \bar{a}_t : t \in I \rangle$ which generates $M$ with a subsequence of $\langle \bar{a}_t : t \in J \rangle$.

By quantifier elimination, $p$ is of the form $\{ R(x,b_\alpha)^{\ii_\alpha} \land x \neq b_\alpha ~:~ \alpha < \kappa \}$ for some $\kappa$, where each 
$\ii_\alpha \in \{ 0, 1 \}$. As $M$ is generated by $\{ \bar{a}_t : t \in I \}$, each $b_\alpha$ 
may be written as 
$\sigma^M_\alpha(\bar{a}_{\bar{t}_\alpha})$ for some $\tau(\Phi)$-term $\sigma_\alpha$ and some $\bar{t}_\alpha \in \inc(I)$. 
This representation may not be unique; choose one, subject to $\lgn(\bar{t}_\alpha)$ being minimal but $\geq 1$. Then we may write $p$ as 
\begin{equation} \label{eq:p}
p(x) = \{ R(x, \sigma^M_\alpha(\bar{a}_{\bar{t}_\alpha}))^{\ii_\alpha} \land 
x \neq \sigma^M_\alpha(\bar{a}_{\bar{t}_\alpha}) : \alpha < \kappa \}. 
\end{equation}
Now working in $N$, consider the set of formulas 
\begin{equation} \label{eq:q}
q(x) = \{ R(x, \sigma^N_\alpha(\bar{a}_{\bar{s}}))^{\ii_\alpha} 
\land x \neq \sigma^M_\alpha(\bar{a}_{\bar{s}})
: \alpha < \kappa, \tpqf(\bar{s}, \emptyset, J) = 
\tpqf(\bar{t}_\alpha, \emptyset, I) \}. 
\end{equation} 
Let's show that $q(x)$ is consistent. It suffices to prove that whenever 
\begin{equation} 
R(x, \sigma^N_\alpha(\bar{a}_{\bar{v}}))^{\ii_\alpha}  \in q \mbox{ and } R(x, \sigma^N_\beta(\bar{a}_{\bar{w}}))^{\ii_\beta}  \in q 
\end{equation}
\emph{if} $\sigma^N_\alpha(\bar{a}_{\bar{v}}) = \sigma^N_\beta(\bar{a}_{\bar{w}})$ \emph{then} $\ii_\alpha = \ii_\beta$. 

Fix for awhile $\alpha$ and $\beta$ and suppose for a contradiction that 
\begin{equation} 
\label{eq:as}
\sigma^N_\alpha(\bar{a}_{\bar{v}_*}) = \sigma^N_\beta(\bar{a}_{\bar{w}_*}) = {b} 
\mbox{ but } \ii_\alpha \neq \ii_\beta. 
\end{equation} 
The contradiction will amount to reducing the collision in $q$ to a collision already in $p$. 
After inessential changes (changing the order of the variables if needed) we may assume 
\begin{equation}
\label{4-a}
 \bar{v}_* = \bar{u} ^\smallfrown \bar{v}, ~ \bar{w}_* = \bar{u} ^\smallfrown \bar{w} 
\end{equation}
where $\bar{u}$, $\bar{v}$, $\bar{w}$ are each in strictly increasing order, so without repetition, and also 
their concatenation $\bar{u}^\smallfrown \bar{v} ^\smallfrown \bar{w}$ is without repetition 
 (so in particular the sequences $\bar{v}$ and $\bar{w}$ share no elements).  For later use, we record the translation of (\ref{eq:as}) via (\ref{4-a}):  
\begin{equation} \label{eqtrans}
  \sigma^N_\alpha(\bar{a}_{\bar{u}^\smallfrown \bar{v} }) = \sigma^N_\beta(\bar{a}_{\bar{u}^\smallfrown \bar{w}}). 
\end{equation}
We now make some observations about the structure of our sequence. 
Let $\bar{u} = \langle u_\ell : \ell < \lgn(\bar{u}) \rangle$ list $\bar{u}$ in increasing order.  
By ``an interval of consecutive elements of $\bar{u}$'' 
we will mean a set of elements $v_i$, $w_j$ which are all 
less than $u_0$, or all greater than $u_{\lgn(\bar{u})-1}$, or all strictly between $u_i, u_{i+1}$ for some 
$0 \leq i < \lgn(\bar{u})-1$. 

Let us justify $(\star)$ that without loss of generality, 
within each interval of consecutive elements of $\bar{u}$, 
all elements of $\bar{v}$ falling in this interval are strictly below all elements of $\bar{w}$ falling in the same interval: formally, 
if $v_k, w_j$ fall in the same interval in $\bar{u}$, then $v_k < w_j$. 
Suppose $(\star)$ fails, that is, suppose we chose our $\bar{v}_*$, $\bar{w}_*$ so that the sum over all intervals of 
the number of elements of $\bar{w}$ less than elements of $\bar{v}$ in each given interval is minimized, but 
we were not able to choose this number to be zero. Then in our example, within at least one interval, say\footnote{if one of the endpoints is 
$+\infty$ or $-\infty$, the same argument applies substituting this notation.} $(u_i, u_{i+1})$ of $\bar{u}$, we have 
elements $v_k, w_j$ such that 
\begin{equation}
\begin{split}
u_i < \{ v \in \bar{v} : ~u_i < v < v_k \} \cup  \{ & w \in \bar{w} :  ~u_i < w < w_j  \} \\ 
& < w_j < v_k <  \\ 
 \{ v \in \bar{v} :  v_k < & ~ v < u_{i+1} \} \cup  \{ w \in \bar{w} : ~ w_j < w < u_{i+1} \} < u_{i+1}.  \\
\end{split}
\end{equation}
where of course, some or all of the sets in the first and third lines may be empty. 
Since $J$ is $\aleph_0$-saturated, we may choose $w^\prime_j, v^\prime_k$ of the same colors as 
$w_j$, $v_k$ respectively, so that  
\[ w_j < v^\prime_k < w^\prime_j < v_k.  \]
Then since quantifier-free type follows from color and order, and we've chosen 
\begin{equation}
\begin{split}
\{ v \in \bar{v} : ~u_i < v < v_k \} \cup  \{ & w \in \bar{w} :  ~u_i < w < w_j  \} \\ 
& < w_j < v^\prime_k < w^\prime_j <  v_k <  \\ 
 \{ v \in \bar{v} :  v_k < & ~ v < u_{i+1} \} \cup  \{ w \in \bar{w} : ~ w_j < w < u_{i+1} \}. \\
\end{split}
\end{equation}
it follows that writing $\bar{v}^\prime$ for $\bar{v}$ in which $v_k$ is replaced by $v^\prime_k$, and 
writing $\bar{w}^\prime$ for $\bar{w}$ in which $w_j$ is replaced by $w^\prime_j$, the quantifier-free type 
does not change, i.e. $\qftp(\bar{w}, \emptyset, J) = \qftp(\bar{w}^\prime, \emptyset, J)$ and 
$\qftp(\bar{v}, \emptyset, J) = \qftp(\bar{v}^\prime, \emptyset, J)$. Since the change is within an interval, it 
follows that 
$\qftp(\bar{u} ^\smallfrown \bar{v} ^\smallfrown \bar{w}, \emptyset, J) =$ 
$\qftp(\bar{u}~ ^\smallfrown {\bar{v}^\prime}~ ^\smallfrown \bar{w}, \emptyset, J)$ and 
$\qftp(\bar{u}~^\smallfrown {\bar{v}^\prime} ~^\smallfrown \bar{w}, \emptyset, J) = 
\qftp(\bar{u}~^\smallfrown {\bar{v}^\smallfrown} ~{\bar{w}^\prime},\emptyset, J)$. 
So by (\ref{eqtrans}) and definition of $\GEM$-model, 
\begin{equation}
\begin{split}
  \sigma^N_\alpha(\bar{a}_{\bar{u}^\smallfrown \bar{v} }) = &  \sigma^N_\beta(\bar{a}_{\bar{u}^\smallfrown \bar{w}}) \\
  \sigma^N_\alpha(\bar{a}_{{\bar{u} }^\smallfrown \bar{v}^\prime }) = &  \sigma^N_\beta(\bar{a}_{\bar{u}^\smallfrown \bar{w}}) \\
   \sigma^N_\alpha(\bar{a}_{\bar{u}^\smallfrown \bar{v} }) = &  \sigma^N_\beta(\bar{a}_{\bar{u}^\smallfrown \bar{w}^\prime}) \\
\mbox{ so by } & \mbox{ transitivity of equality } \\
  \sigma^N_\alpha(\bar{a}_{\bar{u}^\smallfrown \bar{v}^\prime }) = &  \sigma^N_\beta(\bar{a}_{\bar{u}^\smallfrown \bar{w}^\prime}).
\end{split}
\end{equation}
This leaves us with an example of collision (\ref{eq:as}) in which the sum over all intervals of 
the number of elements of $\bar{w}$ less than elements of $\bar{v}$ in each given interval
is strictly smaller than for  $\bar{v}_*$, $\bar{w}_*$.  This contradiction (or this argument repeated) 
proves $(\star)$. 

\br

With $(\star)$, it is now a simple matter to reduce the contradiction (\ref{eq:as}) to a contradiction in  
the type $p$. 
Let $\bar{u}_I$, $\bar{v}_I$ from $I$ realize the same quantifier free type in $J$ 
as $\bar{u}, \bar{v}$, so in particular $\lgn(\bar{u}_I) = \lgn(\bar{u})$ and 
$\lgn(\bar{v}_I) = \lgn(\bar{v})$, and note that $\bar{u}_I, \bar{v}_I$  
are unique by separability of $I$. 
Likewise, let $\bar{u}^\prime_I$, $\bar{w}_I$ from $I$ realize the same quantifier-free type as
$\bar{u}, \bar{w}$ in $J$. Note as $I$ is separated, $\bar{u}^\prime_I = \bar{u}_I$. 
Summarizing, we have 
\begin{equation}
\label{eq15}
 \bar{u}_I, \bar{v}_I, \bar{w}_I  \in {^{\omega >}I} 
\end{equation} 
with $\tpqf(\bar{u}_I ~^\smallfrown \bar{v}_I, \emptyset, J) = \tpqf(\bar{u} ~^\smallfrown \bar{v}, \emptyset, J)$ and 
$\tpqf(\bar{u}_I ~^\smallfrown \bar{w}_I, \emptyset, J) = \tpqf(\bar{u} ~^\smallfrown \bar{w}, \emptyset, J)$. 
Next, let's verify that: 
\begin{equation} \label{e:choose}
\mbox{ we may choose $\bar{v}_2$, $\bar{w}_2 \in {^{\omega >}J}$ such that: }
\end{equation}
\begin{enumerate}
\item[(a)] $(\bar{u}, \bar{v}_2, \bar{w}_2)$ realize the same quantifier-free type in $J$ as 
$(\bar{u}_I, \bar{v}_I, \bar{w}_I)$, i.e. their concatenations have the same type, and the length of each 
piece is the same. 
\item[(b)] $\bar{v}_2, \bar{v}_I$ realize the same quantifier-free type over $\bar{u}^\smallfrown \bar{w}_I$,
\item[(c)] $\bar{w}_2, \bar{w}_I$ realize the same quantifier-free type over $\bar{u}^\smallfrown \bar{v}_I$. 
\end{enumerate}
This is easy to do by $(\star)$. That is, since each interval of $\bar{u}$ is of the form\footnote{again, with appropriate substitutions 
for endpoints of the first and last interval, if needed.}
\[  u_i < \{ v \in \bar{v} : u_i < v < u_{i+1} \} < \{ w \in \bar{w} : u_i < w < u_{i+1} \} < u_{i+1}, \]
using the $\aleph_0$-saturation of $J$, we may choose $\bar{v}_2$, $\bar{w}_2$ interval by interval so that 
\begin{equation} 
\begin{split}  u_i < \{ v \in \bar{v} : u_i < v & < u_{i+1} \} < \\ 
\{ v \in \bar{v}_2&  : u_i < v < u_{i+1} \} \cup \{ w \in \bar{w}_2 : u_i < w < u_{i+1} \} 
\\ 
< ~ & \{ w   \in \bar{w} : u_i < w < u_{i+1} \}  < u_{i+1} 
\end{split}
\end{equation}
noting that the only condition we have imposed is the place in the order where we choose the new elements; the quantifier-free type of 
$\bar{v}_2$, $\bar{w}_2$, and their ordering or collision amongst themselves, 
is free to be determined by (a). As quantifier-free type in $J$ is determined by order and color, this condition on 
ordering is clearly enough to satisfy (b) and (c). 
This completes our justification of (\ref{e:choose}). 

\br
Now we apply indiscernibility (that is, the fact that we are working in a $\GEM$ model) 
and transitivity of equality.  Recalling (\ref{eqtrans}), 
 (\ref{e:choose})(b) and (c) mean that 
\begin{equation}
\label{eqxb} \sigma^N_\alpha(\bar{a}_{\bar{u}^\smallfrown \bar{v}_2}) = \sigma^N_\beta (\bar{a}_{\bar{u}^\smallfrown \bar{w}}) 
\mbox{ and } 
 \sigma^N_\alpha(\bar{a}_{\bar{u}^\smallfrown \bar{v}}) = \sigma^N_\beta (\bar{a}_{\bar{u}^\smallfrown \bar{w}_2}) 
\end{equation}
so by (\ref{eqtrans}) and transitivity of equality, 
\begin{equation}
\label{eqxc}
\sigma^N_\alpha(\bar{a}_{\bar{u}^\smallfrown \bar{v}_2}) =  \sigma^N_\beta (\bar{a}_{\bar{u}^\smallfrown \bar{w}_2}). 
\end{equation}
As we are in a $\GEM$-model, (\ref{eqxc}) along with (\ref{e:choose})(a) implies that 
\[ \sigma^N_\alpha(\bar{a}_{{\bar{u}_I} ~^\smallfrown \bar{v}_I}) = \sigma^N_\beta (\bar{a}_{{\bar{u}_I}~ ^\smallfrown \bar{w}_I}) \]
Recalling that $\ii_\alpha \neq \ii_\beta$, and that $q$ was built directly from $p$ in (\ref{eq:q}), 
(\ref{eqxc}) contradicts the assumption that $p$ is a type. 

This contradiction proves that $q$ is consistent. 

Since $q$ is consistent, in some elementary extension $N^\prime$ of $N$ there is an element which realizes it. 
Let $N^{\prime}_c$ be $N^\prime$ expanded by this constant $c$, and by Skolem functions. (Note that once we add this constant, 
$\bar{a}_J$ may no longer be indiscernible.) 
Apply \ref{d:ramsey-exp} in the case where $N^+ = N^\prime_c$. 
Let $\Psi$ be the template returned, which will be 
proper for $\mk$.  By \ref{d:ramsey-exp} and equation (\ref{eq:q}) above, $\Psi$ has the property that for every $\alpha < \kappa$, if 
$\xr_\alpha = \tpqf(\bar{t}_\alpha, \emptyset, I)$ then the formula $R(c, \sigma_\alpha($--$))^{\ii_\alpha}$ belongs to $\Psi(\xr_\alpha)$. 
Thus, $\Phi \lequ \Psi$ and if $J \in \mk$ and $h$ embeds $I$ into $J$ then $c$ realizes $h(p)$ in $\GEM(J, \Psi)$. 
\end{proof}  

\begin{cor} \label{m20} For every $\Phi \in \up$ there is $\Psi \in \up$ such that $\Phi \leq \Psi$ and 
for every separated $I \in \mk$, $\GEM(I, \Phi) \rstr \{ R_{\operatorname{rg}} \}$ is $\mu$-saturated. 
\end{cor}

\begin{proof} 
This is a counting argument. As we'll appeal to similar arguments again, here let us give the details. 
Let $\Phi$ be given. 
Fix for a moment some $I \in \mk$ which is separated, therefore of size $\leq\mu$. Recall from our hypotheses for the section that 
$\lambda, \mu$ are fixed, $\lambda = \lambda^{<\mu}$ and $|\tau(\Phi)| \leq \lambda$.  So 
$M = \GEM(I, \Phi)$ satisfies $|M| \leq \lambda$.  This means $M$ has $\lambda^{<\mu} = \lambda$ subsets of size $<\mu$, and 
over each such subset $A$, it has at most $2^{|A|} \leq \lambda$ types. Let $\langle p_\alpha : \alpha <\lambda \rangle$ be an enumeration 
of all such types. By induction on $\alpha$, we may build a $\leq$-continuous increasing chain of templates $\Phi_\alpha$, where 
$\Phi_{\alpha+1}$ is the result of applying Claim \ref{m17} in the case $\Phi = \Phi_{\alpha}$, $I$, and $p = p_\alpha$. 
Let $\Psi$ be the union of this chain of templates, recalling \ref{o:inc-seq}. Let $N = \GEM(I, \Psi)$. Then $|N| \leq \lambda$. Moreover, 
$M \rstr_{\{ R_{\operatorname{rg}} \}} ~\preceq~ N \rstr_{\{ R_{\operatorname{rg}} \}}$ and all random graph types over subsets of 
$M$ of size $<\mu$ are realized in $N$. 
By repeating this construction $\lambda$-many times, we obtain a template 
\begin{equation}
\label{eq:phi} 
\Psi_{I,\Phi} 
\end{equation} which is $\geq \Phi$ and has the property that for our given $I$, 
$\GEM(I, \Psi)$ has size $\leq \lambda$ and is saturated for all random graph types over subsets of size $<\mu$. 
(This is already enough for the proof of 
Theorem \ref{t:trg-s} below.)

To find a single $\Psi$ which works for all separated $I$, first note that there are a bounded number of separated $I \in \mk$, up to isomorphism. 
(In fact, there are no more than $\lambda$: any separated $I$ has size $\leq \mu$, and any separated $J$ occurs as a subset of some separated $I$ 
of size $\mu$. So there are no more than $2^\mu$ separated $I$ of size exactly $\mu$, up to isomorphism; each has $\leq 2^\mu$ subsets, 
up to isomorphism, and recall $\lambda \geq 2^\mu$.) 
So we can enumerate all such $I$ as $\langle I_\alpha : \alpha < \lambda \rangle$, each occurring cofinally often, and 
build a $\leq_{\up}$-increasing continuous chain of templates $\Phi_\alpha$, where each $\Phi_{\alpha+1}$ is built as 
$\Psi_{I, \Phi_\alpha}$ from (\ref{eq:phi}) above.  
The union of this chain will be our desired template. 
\end{proof}

Now for the non-saturation half of the argument. In the next claim, recall that ``cut'' means ``unfilled cut''. 
A comment on strategy: 
in any $J \in \mk$,  
the quantifier free type of a tuple is determined by its order type and the colors of its elements. 
$I$ is separated, so each element 
has its own color. 
We consider $I$ as a subset of some saturated $J$ where, say, the order is dense and some fixed color occurs densely often. 
Note that the $f$ we find isn't an embedding of $Y \cup Z$ to $J$, 
because all elements in the range of $f$ will have the same color. 
It would suffice to let $f$ choose a sequence suitably cofinal in each side, all of the same color.

\begin{claim} \label{m23}
Suppose $I \in \mk$ is separated, $\kappa$ is an infinite regular cardinal, and 
$(\langle s_\alpha : \alpha < \kappa \rangle, \langle t_\alpha : \alpha < \kappa \rangle)$ is 
a cut of $I$.  
Then $\GEM(I, \Phi) \rstr \{ <_{dlo} \}$ is not $\kappa^+$-saturated, and in fact omits the type 
\[ q(x) = \{ (~ s_\alpha ~<_{dlo} ~ x  ~ <_{dlo} t_\alpha ~) : \alpha < \kappa \} .\]
\end{claim}

\begin{proof} 
We'll prove the a priori stronger claim that for some sufficiently saturated $J$ with $I \subseteq J$, $p$ is not 
realized in $EM(J, \Phi)$. This suffices under our global assumption (\ref{c:nice}) that $EM(I, \Phi) \preceq EM(J, \Phi)$. 

Recall assumption \ref{k4}(d) which says that our templates ``represent'' $<_{\operatorname{dlo}}$:  
when $s < t$ are from the index model then 
$a_s \dl a_t$. 
Let $Y = \{ s_\alpha : \alpha < \kappa \}$ and 
let $Z = \{ t_\alpha : \alpha < \kappa \}$.   
Observe that we can find some sufficiently saturated $J \in \mk$, containing $I$, and a function $f$ such that:
\begin{enumerate} 
\item[(a)] $f$ is a function from $Y \cup Z$ to $J$
\item[(b)] if $s \in Y$ then $s <_{J} f(s) <_{J} (Y)_{>s} := \{ d \in Y : s <_{I_1} d \}$. 
\item[(c)] if $t \in Z$ then $(Z)_{<t} <_{J} f(t) <_{J} t$
\item[(d)] if $s \in Y$, $t \in Z$ then $f(s), f(t)$ realize the same quantifier free type in $J$ over 
$(J)_{<s} \cup (J)_{>t}$. 
\end{enumerate}
Note that $p(x)$ implies $\{ (a_{f(s)} <_{dlo} x) \land (x <_{dlo} a_{f(t)}) : (s, t) \in Y \times Z \}$. 
In $\GEM(J, \Phi)$ if there were $a = \sigma(\bar{a}_{\bar{u}})$ realizing $p$, choose $s \in Y$ and $t \in Z$ such that 
$[s, t] \cap \range(\bar{u}) = \emptyset$, which possible simply because $\bar{u}$ is finite. 
Then because we have assumed $a$ realizes $p$, $a$ must satisfy the formula $(a_{f(s)} <_{dlo} x)$ and also the formula $\neg(a_{f(t)} <_{dlo} x)$. 
This contradicts (d), so completes the proof. 
\end{proof}

We now prove Theorem \ref{t:trg-s} from the beginning of the section. 

\begin{proof}[Proof of Theorem \ref{t:trg-s}.] 
Recall \ref{o:sect3} and from \ref{h:mu} the assumption $\lambda = \lambda^{<\mu}$. 
It suffices to show that for every $\kappa^+ \leq \mu$, 
there are $\Phi \in \up$ and $I \in \mk$ such that for $M = \GEM_{\tau(T_*)}(I, \Phi)$, we have that 
$M \rstr_{\{ R_{\operatorname{rg}} \}}$ is 
$\mu$-saturated but $M \rstr_{\{ <_{dlo} \}}$ is not $\kappa^+$-saturated.   
Choose $I \in \mk$ which is separated and has a $(\kappa, \kappa)$-cut  (note that there is a cardinality 
restriction on separated $I$, but in our case $\kappa < \mu$). Let $\Phi$ be from \ref{m20} and apply it to the  
selected $I$. By \ref{m20} and \ref{m23}, $M_{\tau(T_*)}(I, \Phi)$ is as desired.  
\end{proof}

\begin{concl}
$\trg$ is minimum among the $($complete, countable$)$ unstable theories in $\tlf^*_1$ and does not belong to the maximum class. 
\end{concl}

\begin{proof}
By Lemma \ref{l:trg} and Theorem \ref{t:trg-s}.
\end{proof}

\section{$\tfeq$ is minimal among the non-simple theories} \label{s:tfeq}

\begin{theorem} \label{t:tfeq}
$\tfeq$ is $\tlf^*_1$-minimum among the complete countable non-simple theories. 
\end{theorem}

\begin{proof}
Let $T_1$ be a complete, countable, non-simple first order theory. Without loss of generality, $T_1$ has $TP_2$ and not 
$SOP_2$, as $SOP_2$ is already sufficient for maximality, Fact \ref{maxl-fact} above. We'll build on the 
proof of Lemma \ref{l:trg}. 

Let $\vp(\bar{x}, \bar{y})$ be a formula which has $TP_2$ for $T_1$. (In what follows, we'll write as if $\ell(x) = \ell(y) =1$, 
but this is again only for simplicity of notation.) That is, in some model of $T_1$, there is an array $\{ a_{i,j} : i < \omega, j<\omega \}$ 
of elements of $M$ (i.e. of ${^{\ell(y)}M}$) such that for any $X \subseteq \omega \times \omega$, 
$\{ \vp(x,a_{i,j}) : (i,j) \in X \}$ is consistent if and only if $(i,j) \in X~\land~(i^\prime,j) \in X \implies i=i^\prime$, i.e. $X$ does not 
contain more than one element from each column. 

We'll assume that $T_1$, $\trg$, $\tfeq$, and $Th(\mathbb{N}, <)$ have disjoint signatures.

Let $M$ be a countable model whose theory $T$ satisfies:
\begin{enumerate}[(a)]
\item $M$ expands $(\mathbb{N}, <)$.

\item $M \rstr \tau(\tfeq)$ is a countable model of $\tfeq$. 

That is, there is a unary predicate $P^M$ which is countably infinite. $P^M$ is partitioned into two infinite sets, 
$P^M_0$ and $P^M_1$. On $P^M_1$, there is an equivalence relation $E^M_0$ which has infinitely many classes, all of which are infinite. 
Finally, there is a function $F^M_0: P^M_0 \times P^M_1 \rightarrow P^M_1$ which essentially chooses, for each $a$ in the set $P_0$, 
a path through the equivalence classes. More formally, for each $(a, b) \in P^M_0 \times P^M_1$, 
$E^M_0 (~ F^M_0(a,b), b~)$ and for any finitely many $b_1, \dots, b_n$, $b^\prime_1, \dots, b^\prime_m$ from $P^M_1$ which are pairwise 
$E^M_0$-inequivalent, there is $a \in P^M_0$ such that 
\[  M \models \bigwedge_{1\leq i \leq n} F_0(a,b_i) = b_i ~\land~ \bigwedge_{1 \leq j \leq m} \neg F_0(a,b^\prime_j) = b^\prime_j. \]

\item $M \rstr \tau(T_1)$ is a countable model of $T_1$, 
containing a sequence $\langle b_i : i < \omega \rangle$ satisfying the 
conclusion of Claim \ref{c:indep}.  

The domain of this model is $Q^M$. $S^M$ is a binary relation 
with $T \vdash S \subseteq Q \times Q$ and $T \vdash$ ``$S$ has $TP_2$'', that is, there is an array $\{ a_{i,j} : i < \omega, j<\omega \}$ 
in $Q^M$ as described above. (\emph{As before, we just let $S$ name $\vp$. Note that the sequence $\bar{b}$ will \emph{not} 
necessarily be a sequence on which $\vp$, or $S$, has $TP_2$; we simply need it to be a sequence on which $\vp$, so $S$, has the 
independence property, which is fine as $TP_2$ implies $IP$.})

\item $F^{M}$ is an injective function from $\mathbb{N}$ into $Q^{M}$ such that for every $a \in Q^M$, 
 for some truth value $\trv$, for every $n$ large enough, $M \models \vp[a, F^{M}(n)]^{\trv}$. 
\\ (\emph{It suffices to let $F^M(n) = b_n$ where $\langle b_n : n < \omega \rangle$ is from Claim \ref{c:indep}.})

\item $G^M$ is a one to one and onto function from $P^M_1$ to $\{ a_{i,j} : i<\omega, j<\omega \} \subseteq Q^M$ 
which respects consistency and inconsistency in the natural way, \emph{i.e.} 
such that for $b, b^\prime \in P^M_1$, $\{ S^M(x,G^M(b)), S^M(x,G^M(b^\prime)) \}$ is inconsistent. 
\\ (\emph{Note that $T$ will record that for all finite $k$, if $b_1, \dots, b_k \in P^M_1$ then} 
\[ \{ S(x, G(b_1)), \dots, S(x,G(b_k)) \} \] \emph{ is consistent in $M$ iff the $b_i$ are pairwise $E^M_0$-inequivalent.})

\item $G^M_*$ is a one to one and onto function from $P^M_1$ to $\mathbb{N}^M$. 

\noindent (\emph{Note that $T$ will record that for every $n \in \mathbb{N}^M$, for every definable subset 
$U \subseteq n$, if $\{ (G^{-1}_*)^M(a) : a \in U  \}$ are pairwise $E^M_0$-inequivalent, then the set of 
$\{ F_0(x,G^{-1}(a)) = a : a \in U \} \cup \{ F_0(x,G^{-1}(b)) : b < n, b \notin U \}$ is consistent and moreover 
realized by an element of $P^M$. Note that definable means with parameters in $\ml(\tau(T))$.})

\item Finally, though we won't need to refer to the rest by name,  
for every instance of the word ``infinite'' in the above catalogue, add a new function symbol interpreted as a bijection 
between $\mathbb{N}$ and the given infinite set. In the case of the equivalence relation, it will be a parametrized family of functions. 
\end{enumerate}

\br
Now let $N$ be any model of $T$. It will suffice to prove that if $N \rstr \tau(T_1)$, or really just $N \rstr \{ Q, S \}$, is $\mu$-saturated, 
then the following are true. Since our theory simply expands that described in the proof of Lemma \ref{l:trg} (in the case, say, where 
$T_1 = \trg$) we have by the same proof that (1) and (2) where: 

\br
\noindent (1) \emph{$\mathbb{N}^{N}$ has cofinality $\geq \mu$.}  

{Note: this is the only place we use the sequence from Claim \ref{c:indep}.}

\br
\noindent (2) \emph{$N$ satisfies $<\mu$-regularity, meaning that every set of $<\mu$ elements is contained in some pseudofinite set.}

\br
\noindent (3) \emph{$N \rstr \tau(\tfeq)$ is $\mu$-saturated.} 

Suppose $P^M$ is not $\mu$-saturated. In the most interesting case, there is an omitted $1$-type $p$ of cardinality $< \mu$ of the form: 
\[ \{ ( F_0(x, b_\alpha) = b_\alpha )^{\operatorname{if } \eta(\alpha)}  : \alpha < \alpha_* <\mu \}   \]
for some $\eta \in {^{\alpha_*}2}$ and $\langle b_\alpha : \alpha < \alpha_* \rangle$ pairwise $E^N_0$-inequivalent. 
Invoking the bijection $G^N_*$ from item (f) from $Q^N$ onto $\mathbb{N}^N$, we know that by item (1), 
the image of $\{ a_\alpha : \alpha < \alpha^* \}$ is bounded in $\mathbb{N}^N$, say by $b_*$. 

As before we translate to a type in $T_1$. 
Let $a_\alpha = G^N(b_\alpha)$ for $\alpha < \alpha_*$. 
Then $\{ a_\alpha : \alpha < \alpha_* \}$ is a subset of $Q^N$, and 
$\{ S(x,a_\alpha)^{\operatorname{if} \eta(\alpha)} : \alpha < \alpha_* \}$ is consistent by our definition of $T$.  
By our assumption that $N \rstr \{ Q, S \}$ is saturated, this type is realized, say by $d$. 

And just as before, we have that $U = \{ b \in P^N_1 :  G^N_*(b) < b_*,  S^N(d, G^N(b)) \}$ is a first-order definable subset of $N$ (with parameters).  By our choice of $T$, as explained in the comment to item (f),  this is enough to show 
$p$ is realized in $P^M$, which proves (3). This 
completes the proof of the theorem.  
\end{proof}

\begin{cor}
$\tfeq$ is $\tlf^*_{\aleph_1}$-minimum among the complete countable non-simple theories. 
\end{cor}

\section{Non-simple theories are not below simple theories} \label{s:simple}
\setcounter{theoremcounter}{0}

In this section we prove 
Theorem \ref{t:simple}, which says that non-simple theories are not below simple theories 
in the interpretability ordering $\tlf^*$. 
The precedent 
is the main theorem of \cite{MiSh:1030}, Theorem 8.2 there, 
which shows that assuming existence of a supercompact cardinal $\sigma$, 
there exist regular ultrafilters which saturate all simple theories of size $< \sigma$ 
and do not saturate any non-simple theories. This implies that under a large cardinal hypothesis,  
$\neg ( T_1 \tlf^*_{\aleph_1} T_0)$. The proof we give here is in ZFC, so is a strict improvement on this quotation. 

\begin{theorem} \label{t:simple} 
Let $T_0$ be any simple theory and $T_{1}$ any non-simple theory. Then 
\[ \neg (T_{1} \tlf^*_1 T_0). \]
\end{theorem}

We will prove the theorem at the end of the section, after several lemmas. 

\begin{obs}
Since any theory which is not simple has either $TP_2$ or $SOP_2$, and any theory with $SOP_2$ is $\tlf^*$ maximal by 
$\ref{maxl-fact}$, it will suffice to prove the theorem in the case where $T_1$ has $TP_2$.  In fact, by Theorem \ref{t:tfeq} above, 
it would suffice to prove the theorem in the case where $T_1 = \tfeq$. 
\end{obs}

\begin{hyp} \label{n2}  \emph{ In this section, }
\begin{enumerate}
\item $T_0$ is a simple theory. $($We can use many such theories simultaneously.$)$
\item $T_1$ is a non-simple theory with $TP_2$. 
\item $\kappa = \cf(\kappa) \geq \kappa(T_0)$.
\item $\kappa \leq \mu$, $\lambda = \lambda^{<\mu}$. 

\item $T_*$ is a theory which interprets both $T_0$ and $T_1$, i.e. a potential candidate for showing 
$T_0 \tlf^* T_1$.  We assume $(\ref{s:sum}$ above$)$ $T_*$ has Skolem functions. 

\item $F_*$ is a binary function symbol of $\tau(T_*) \setminus \tau(T_0) \setminus \tau(T_{1})$ and there is an identification between 
some formula of $T_1$ with $TP_2$ and the graph of $F$ in the sense that:
\begin{enumerate}
\item  for any $M \models T_0$, $M \models$ ``$F_*$ is a $2$-place function such that any finite function is represented by some $F(-,a)$.''  
\item  if $M \rstr \{ F_* \}$ omits a type, then $M \rstr \tau(T_1)$ omits one of the same size. 
\end{enumerate}

\item $\mk = \mk^{tr}_\kappa$, the class of normal trees with $\kappa$ levels, lexicographic order, tree order and 
predicates for levels, see $\ref{n5}$ above. 

\item $\Upsilon$ is the class of templates proper for $\mk$ which satisfy our global hypotheses $\ref{c:nice}$ and also satisfy: 
for every $I \in \mk$ and $\Phi \in \Upsilon$, $\GEM(I, \Phi) \models T_*$. 

\item $\leq = \leq_\Upsilon$ is the natural order on this class, as in \ref{d:order}. 
\end{enumerate}
\end{hyp}

\begin{obs} \label{o:inc-seq}
$\Upsilon$ is closed under unions of increasing sequences of length $\leq \lambda$ $($and more but this is 
all we need here$)$. 
\end{obs}

A comment on strategy. 
First, in \ref{n11}, we'll show that we may increase the given template $\Phi$ to $\Psi$ to ensure types have a finite 
satisfiability property. Note that Claim \ref{n11}(2) tells us types in $\GEM(..., \Phi)$ are finitely satisfiable 
in $\GEM(...,\Psi)$; a simple induction in Claim \ref{n12} is needed to use the 
same template in both halves of the statement. The saturation argument, Lemma \ref{n14}, 
depends on showing that if we take a type over (the Skolem hull of 
those parts of the skeleton whose indices lie in) a single branch and look at many copies of such a type, 
their union is consistent. In that proof the independence theorem plays a key role. Its use will be 
justified by finite satisfiability, from \ref{n11} and \ref{n12}.  The non-saturation argument is Claim \ref{n17}.

\begin{claim} \label{n11}
For every $\Phi \in \up$, there is $\Psi \in \up$ such that:
\begin{enumerate}
\item $\Phi \leq \Psi$
\item for every standard $I \in K$ and $\eta \in I$ of level $i < \kappa$, every type 
 of $\tau(T)$ which 
$\GEM(I^{\geq \eta} \cup I^{\leq \eta}, \Phi)$ 
realizes over 
$\GEM(I^{\perp \eta} \cup I^{\leq \eta}, \Phi)$ 
inside 
$\GEM(I, \Phi)$ is finitely satisfiable in $\GEM(I^{\leq \eta}, \Psi)$ where:
\\ $I^{\perp \eta} = \{ \eta \in I :  \neg(\eta \tlf \nu) \}$, 
\\ $I^{\geq \eta} = \{ \nu \in I : \eta \tlf \nu \}$, 
\\ $I^{\leq \eta} = \{ \nu \in I : \nu \tlf \eta \}$. 
\end{enumerate} 
\end{claim}
\setcounter{equation}{0}

\begin{proof}  
To begin, let's carefully choose $I_0, I_1 \in \mk$. 
Towards this, fix $J_0$ to be any infinite $\aleph_0$-saturated linear order.  
Let $J_1$ be the linear order given by $J_0 \times \mathbb{Q}$, with the usual (lexicographic) order.

Let $I_1 \in \mk$ be the index model whose domain is 
${^{\kappa>}J_1}$. Then $I_1$ is a tree of sequences [of pairs, though we can't refer to the pairing 
in $\tau(\mk)$] with predicates $P_i$ naming level $i$ for $i < \kappa$, the tree order $\tlf$, and the lexicographic 
order $\lex$, i.e. lexicographic order on the tree.  Let $I_0 = {^{\kappa >}(J_0 \times \{ 0 \} )} \subseteq I_1$, the sequences of pairs with 
second coordinate constantly $0$.  Let $M_1 = \GEM(I_1, \Phi)$ with skeleton $\ma = \{ \bar{a}_\eta : \eta \in I_1 \}$ and let 
$M_0 = \GEM(I_0, \Phi)$ with skeleton $\ma \rstr I_0$. 

This construction accomplishes: 
 
\br
\begin{enumerate}
\item[(a)] $I_0$ is $\aleph_0$-saturated for $\mk$ (so later we may apply \ref{d:ramsey-exp}). 
\item[(b)] $M_0 = \GEM(I_0, \Phi) \preceq M_1 = \GEM(I_1, \Phi)$, immediate by $I_0 \subseteq I_1$, see \ref{c:nice}. 
\item[(c)] $M_1$ acts like a larger saturated model around $M_0$ in a sense we now explain. 
\end{enumerate}

\br
\noindent  
Working in $M_1$, let's ``pad'' $M_0$ by building in witnesses to finite satisfiability, as follows.
Define a new set of function symbols 
\begin{equation} 
\mcf = \{ F_{i+1, j, \nu} : i + 1< j < \kappa, \nu \in {^{[i+1, j]}(\mathbb{Q} \setminus \{ 0 \})} \}. 
\end{equation} 
Let $M^+_1$ be $M_1$ expanded to a model of $\tau(\Phi) \cup \mcf$ in the 
following way.\footnote{Informally, for every $\eta \in I_0$ of successor length $i+1$, and every 
given sequence $\nu$ of $j$ additional non-zero rationals, the function $F_{i+1,j,\nu}$ sends $\bar{a}_{\eta}$ 
to $\bar{a}_{\rho}$ where $\rho$ is obtained 
by concatenating onto $\eta$ a sequence of $j$ additional elements whose first coordinate just repeats the last first coordinate 
of $\eta$ and whose second coordinates are those given by $\nu$. The reason to use $i+1$ is to have a last first coordinate to repeat.}  
For every 
$i + 1< j < \kappa$, and for every $\nu \in {^{[i+1, j]}(\mathbb{Q} \setminus \{ 0 \})}$, 
expand $M_1$ by defining $F_{i+1, j, \nu}$ to be the function 
with domain $\{ \bar{a}_{\eta} : P_{i+1}(\eta) \} = \{ \bar{a}_\eta : \eta \in I_0, \ell(\eta) = i+1 \}$ such that  
$F_{i+1,j,\nu}(\bar{a}_\eta) = \bar{a}_\rho$ when $\eta \triangleleft \rho \in {^{j}(J_1)}$ and 
$(\forall j)(i+1\leq i^\prime < j \implies \rho(i^\prime) = (t, \nu(i^\prime))$, where $t$ is such that 
$\eta(i) = (t, 0)$.  

Let $M^\star$ be the submodel of $M^+_1$ generated by $\{ \bar{a}_\eta : \eta \in I_0 \}$.  
We now argue that $M^\star$ has the following key property.  

\begin{subclaim} \label{sc}
For every quantifier free formula $\vp(\bar{x}_0, \dots, \bar{x}_{k-1}, \bar{y}_0, \dots, \bar{y}_{m-1})$ 
of $\tau(\Phi)$, 
every $\eta \in I_0$,  
every $\eta_0, \dots, \eta_{m-1} \in {I_0}^{\perp \eta } \cup {I_0}^{\leq \eta}$ 
and every 
$\eta^*_0, \dots, \eta^*_{k-1} \in {I_0}^{\geq \eta } \cup {I_0}^{\leq \eta}$  
there exist 
function symbols $F_0, \dots, F_{k-1} \in \mcf$ such that  
\[ \mbox{ if } M^\star \models \vp[\bar{a}_{\eta^*_0}, \dots, \bar{a}_{\eta^*_{k-1}}, \bar{a}_{\eta_0}, \dots, \bar{a}_{\eta_{m-1}} ] \]
\[ \mbox{ then } M^\star \models \vp[F_0(\bar{a}_{\eta}), \dots, F_{k-1}(\bar{a}_{\eta}), 
\bar{a}_{\eta_0}, \dots, \bar{a}_{\eta_{m-1}}].\] 
Moreover, the choice of functions is an invariant of the set of 
types 
\[ \{\tpqf(\langle \eta, \eta_0, \dots, \eta_{m-1}\rangle, \emptyset, I_0), \tpqf(\langle \eta, \eta_0, \dots, \eta_{m-1}, \eta^*_0, \dots, \eta^*_{k-1}\rangle, \emptyset, I_1) \}. \]
\end{subclaim}

\begin{proof}[Proof of Subclaim \ref{sc}.] 
We unwind the definitions. As $M^\star$ is a submodel of $M^+_1$ and $\vp$ is quantifier free in $\tau(\Phi)$, 
\begin{align*} 
 M^\star \models \vp[\bar{a}_{\eta^*_0}, \dots, & \bar{a}_{\eta^*_{k-1}}, \bar{a}_{\eta_0}, \dots, \bar{a}_{\eta_{m-1}} ] \\
\iff & 
 M^+_1 \models \vp[\bar{a}_{\eta^*_0}, \dots, \bar{a}_{\eta^*_{k-1}}, \bar{a}_{\eta_0}, \dots, \bar{a}_{\eta_{m-1}} ].  
\end{align*}
As for the elements in the index model, quantifier free type depends only on level, tree-order, and lexicographic order, 
for each $\ell < k$ we may find $i_\ell, j_\ell, \nu_\ell, \rho_\ell$ such that 
first, $i_\ell + 1< j_\ell < \kappa$ and $\nu_\ell \in {^{[i_\ell+1, j_\ell]}(\mathbb{Q} \setminus \{ 0 \})}$, 
second, $F^{M^+_1}_{i_\ell+1,j_\ell,\nu_\ell}(\bar{a}_\eta) = \bar{a}_\rho$ for $\ell < k$, 
and third,\footnote{Note that in (\ref{eq:types}) the $\eta_i$'s are elements of $I_0 \subseteq I_1$ 
while the $\rho_\ell$'s are just elements of $I_1$, 
the index model for $M_1$.  Elements of the form $\bar{a}_{\rho_\ell}$ belong to the skeleton of $M_1$, 
and a fortiori to the expanded model $M^+_1$.  These elements \emph{also} belong to the smaller model 
$M^+$ by virtue of being equal to $F^{M^+_1}_{i_\ell+1,j_\ell,\nu_\ell}(\bar{a}_\eta)$. However, 
it would be misleading to say ``$\bar{a}_{\rho_\ell} \in M^+$'' because the notation would suggest it is an 
element of the skeleton, which it is not since $\rho_\ell \notin I_0$.}
\begin{align}
\label{eq:types}
\begin{split} 
\qftp(\eta, \eta_0, \dots, \eta_{m-1}, \eta^*_0, \dots, \eta^*_{k-1}, \emptyset, I_1) =  \\ 
\qftp(\eta, \eta_0, \dots, \eta_{m-1}, \rho_0, \dots, \rho_{k-1}, \emptyset, I_1).  
\end{split}
\end{align}
Now by definition of $\GEM$ model, since the skeleton is quantifier-free indiscernibe, 
\begin{align*} 
 M^+_1 \models \vp[\bar{a}_{\eta^*_0}, & \dots, \bar{a}_{\eta^*_{k-1}}, \bar{a}_{\eta_0}, \dots, \bar{a}_{\eta_{m-1}} ] \\ 
\iff & 
M^+_1 \models \vp[\bar{a}_{\rho_0}, \dots, \bar{a}_{\rho_{k-1}}, \bar{a}_{\eta_0}, \dots, \bar{a}_{\eta_{m-1}} ]. 
\end{align*}
By our choice of $\rho_\ell$, the last equation above holds if and only if 
\begin{equation} 
\label{eq:8}
M^+_1 \models \vp[F_{i_0+1,j_0, \nu_0} (\bar{a}_{\eta}), \dots, F_{i_{k-1}+1, j_{k-1}, \nu_{k-1}}(\bar{a}_{\eta}), 
\bar{a}_{\eta_0}, \dots, \bar{a}_{\eta_{m-1}}]. 
\end{equation}
so recalling the definition of $M^\star$ and the fact that $\vp$ is quantifier free, (\ref{eq:8}) holds if and only if 
\begin{equation} 
\label{eq:9}
M^\star \models \vp[F_{i_0+1,j_0, \nu_0} (\bar{a}_{\eta}), \dots, F_{i_{k-1}+1, j_{k-1}, \nu_{k-1}}(\bar{a}_{\eta}), 
\bar{a}_{\eta_0}, \dots, \bar{a}_{\eta_{m-1}}] 
\end{equation}
which proves the subclaim. 
\noindent\emph{Proof of Subclaim \ref{sc}.}
\end{proof}

Before continuing, we record the following immediate corollary to the proof of Subclaim \ref{sc}. 
We'll use $x$'s and $y$'s for 
arbitrary elements of $\tau(T)$-models and $s$'s and $t$'s and $v$'s for arbitrary elements of index models. 

\begin{subclaim} \label{sc2}
Let $\vp(\bar{x}_0, \dots, \bar{x}_{k-1}, \bar{y}_0, \dots, \bar{y}_{m-1})$ be a quantifier free formula 
of $\tau(\Phi)$. Suppose $\xr(t,t_0, \dots, t_{k-1}, s_0,\dots, s_{m-1}) \in D_{\operatorname{qf}}(I_0)$ is a type which 
satisfies    
${\xr \vdash \mbox{`} t \triangleleft t_\ell ~\lor~ t_\ell \trianglelefteq t}$'' for each $\ell < k$ and 
$\xr \vdash$ ``$s_i \perp t ~ \lor ~ s_i \tlf t$'' for each ${i < m}$.  
Then there exist functions 
$F_0, \dots, F_{k-1} \in \mcf$ such that \\ the formula $\psi = \psi_\xr(x,x_0,\dots,x_{k-1},y_0,\dots,y_{m-1})$ given by 
\[ \vp(x_0, \dots, x_{k-1}, y_0, \dots, y_{m-1}) \implies 
\vp(F_0(x), \dots, F_{k-1}(x), y_0, \dots, y_{m-1}) \]
belongs to $\tpqf(\bar{a}_{\bar{v}}, \emptyset, M^+_1)$ for any $\bar{v}$ from $I_0$ realizing $\xr$. 
\end{subclaim}

We are ready to find $\Psi$. 
Expand $M^\star$ 
to a model $M^{\star\star}$ whose theory has Skolem functions. 
By the Ramsey property \ref{d:ramsey-exp} applied with $I_0$, $M_0$ and $\ma \rstr I_0$, and $M^{\star\star}$ here for 
$J$, $M$ and $\ma$, and $N^+$ there, there exists a template $\Psi \geq \Phi$ which is proper for $I_0$ and which has the property that 
for each $\xr$ satisfying the hypothesis of Subclaim \ref{sc2}, the formula $\psi_\xr$ from that Subclaim 
belongs to $\Psi(\xr)$. 

Let us verify that $\Psi$ satisfies the property of the claim. 
Let $I \in \mk$ be any standard index model. 
Let $N = \GEM(I, \Psi)$.  
Let a quantifier-free formula $\theta(\bar{x}, \bar{y})$ of $\ml(\tau(T))$ be given; this will suffice for the claim  
as $T_\Phi$ has Skolem functions.  Note that by definition of $\leq_\Upsilon$, $\GEM_{\tau(T)}(I, \Phi) \preceq \GEM_{\tau(T)}(I, \Psi)$. 
Suppose $N \models \theta[\bar{b}, \bar{c}]$ where for some $\eta \in I$, 
$\bar{b}$  
is a finite sequence of elements of $\GEM(I^{\geq \eta} \cup I^{\leq \eta}, \Phi)$ 
and $\bar{c}$ 
is a finite sequence of elements of $\GEM(I^{\perp \eta} \cup I^{\leq \eta}, \Phi)$.
We would like to find $\bar{b}^\prime$ from $\GEM(I^{\leq \eta}, \Psi)$ such that $N \models \theta[\bar{b}^\prime, \bar{c}]$. 
By definition of $\GEM$-model, 
there are elements $\eta^*_0, \dots, \eta^*_{k-1} \in I^{\geq \eta} \cup I^{\leq \eta}$ and $\tau(\Phi)$-terms  
$\sigma^*_0, \dots, \sigma^*_{\ell-1}$ such that 
\[ \langle \sigma^*_0(\bar{a}_{\eta^*_0}, \dots, \bar{a}_{\eta^*_{m-1}}), \dots, 
 \sigma^*_{j-1}(\bar{a}_{\eta^*_0}, \dots, \bar{a}_{\eta^*_{m-1}})\rangle = \bar{b} \] 
and also elements $\eta_0, \dots, \eta_{m-1} 
\in I^{\perp \eta} \cup I^{\leq \eta}$ and $\tau(\Phi)$-terms $\sigma_0, \dots \sigma_{j-1}$ such that 
\[ \langle {\sigma}_0(\bar{a}_{\eta_0}, \dots, \bar{a}_{\eta_{k-1}}), \dots, 
{\sigma}_{\ell-1}(\bar{a}_{\eta_0}, \dots, \bar{a}_{\eta_{k-1}})\rangle = \bar{c}. \] 
Let $\vp(x_0, \dots, x_{k-1}, y_0, \dots, y_{m-1})$ be the quantifier-free formula 
equivalent to 
\\ $\theta(~\sigma^*_0(x_0, \dots, x_{k-1})$, $\dots$, 
$\sigma^*_{\ell-1}(x_0$, $\dots$, $x_{k-1})$, $\sigma_0(x_0$, $\dots$, $x_{m-1})$, $\dots$, 
$ \sigma_{j-1}(x_0$, $\dots$, $x_{m-1})$ $~)$. By construction it is still a $\tau(\Phi)$-formula. 
Let 
\[ \mathfrak{r} = \tpqf(\eta~^\smallfrown \eta_0~^\smallfrown\cdots ^\smallfrown \eta_m ~^\smallfrown \eta^*_0~^\smallfrown\cdots ^\smallfrown \eta^*_{k-1}, \emptyset, I). \]
Recall $I_0$ and $M^\star$ from earlier in the proof. 
Because $I_0$ was $\aleph_0$-saturated, there is some sequence $\bar{\eta}$ of elements of $I_0$ realizing $\mathfrak{r}$. Because $\Psi \geq \Phi$, $M_1 = \GEM(I_1, \Phi) \preceq \GEM(I_1, \Psi)$ and recall that $M^\star$ is a submodel of 
$M^+_1$, so a fortiori $M^\star \rstr_{\tau(\Phi)} \subseteq M_1$. 
As $\vp$ is a quantifier-free $\tau(\Phi)$-formula, it must be that $M^\star \models \vp[\bar{a}_{\bar{\eta}}]$. 
Apply Subclaim \ref{sc2} to finish the proof. (Note: we've written finitely satisfiable in ``$I^{\leq\eta}$,'' but we've 
used ``$I^{\eta}$.'')

\noindent\emph{Proof of Claim \ref{n11}}.  
\end{proof}

\begin{rmk}  
The proof of Claim $\ref{n11}$ did not use any of the assumptions on $T$, in particular it did not use the simplicity of $T$; 
 so this is also true of Corollary $\ref{n12}$. 
\end{rmk}

\begin{cor} \label{n12} 
Let $I$ be standard. For every $\Phi \in \up$, there is $\Psi \in \up$ with 
$\Phi \leq \Psi$ such that  
every type of $\tau(T)$ which 
$\GEM(I^{\geq \eta} \cup I^{\leq \eta}, \Psi)$ 
realizes over 
$\GEM(I^{\perp \eta} \cup I^{\leq \eta}, \Psi)$ 
inside $\GEM(I, \Psi)$ is finitely satisfiable in $\GEM(I^{\leq \eta}, \Psi)$. 
\end{cor}

\begin{proof} 
Let $\Phi_0 = \Phi$. 
Choose $\Phi_n$ by induction on $1 \leq n < \omega$ to be the result of applying Claim \ref{n11} with $\Phi = \Phi_{n-1}$. 
Then $\Psi = \bigcup_n \Phi_n$ is the desired template, and $\Phi \leq \Psi \in \Upsilon$ recalling  Observation \ref{o:inc-seq}.
\end{proof}

Now we will use the hypothesis that $T$ is simple. 

\begin{lemma} \label{n14}
Let $I$ be standard with universe ${^{\kappa > }{\{ 0 \}}}$. For every $\Phi \in \up$, there is $\Psi \in \up$ with 
$\Phi \leq \Psi$ such that $M = \GEM_{\tau(T)}(I, \Psi)$ is $\mu$-saturated. 
\end{lemma}

\begin{proof}
Let $I$ and $\Phi$ be given.  Without loss of generality $\Phi$ 
satisfies the conclusion of 
Corollary \ref{n12}.  Let $M = \GEM_{\tau(T)}(I, \Phi)$. It will suffice to show that if $p \in \ts(M)$ is a type over a 
set of size $<\mu$ then\footnote{We won't really use the size of $p$ when realizing a single type, but 
just as in \ref{m20}, it's important to keep track of size when iterating to produce saturation.} we can find $\Psi \geq \Phi$ such that $p$ is realized in 
$\GEM(I, \Psi)$. We can then iterate to obtain the template producing a $\mu$-saturated model just as in Claim \ref{m20}.  

The first use of simplicity will be not forking over a small set. For $i<\kappa$, let 
$M_i = \GEM_{\tau(T)}(\{ \eta \in I : \lgn(\eta) < i) \}, \Phi)$, so the sequence 
$\langle M_i : i \leq \kappa \rangle$ is $\preceq$-increasing continuous and its union $M_{\kappa} = M$.
As $T$ is simple and complete and $\kappa \geq \kappa(T)$, there is $i_* < \kappa$ such that 
$p$ dnf over $M_{i_*}$. For simplicity, we may assume $i_*$ is a successor. 

Towards finding $\Psi$, we move to work in a saturated index model. Let 
$\chi$ be infinite so $J = {^{\kappa >}\chi}$ is $\aleph_0$-saturated.  Let $N = \GEM_{\tau(T)}(J, \Phi)$, so $M \preceq N$. 
Let $\ma$ denote the skeleton of $N$, extending that of $M$.  
For every $\eta \in {^\kappa \chi}$ let $h_\eta$ be the canonical isomorphism from $I$ [recalling it is a single branch] onto 
$J_\eta = J \rstr \{ \eta \rstr i : i < \kappa \}$. Let $\hat{h}_\eta$ be the induced 
isomorphism from $M$ to $N_\eta = \GEM_{\tau(T)}(J_\eta, \Phi)$ and let $p_\eta = \hat{h}_\eta(p)$. 

We may likewise write these models as unions of chains: 
let $N_{\eta, i} = \GEM_{\tau(T)}( \{ \eta \in J_\eta  : \lgn(\eta) < i) \}, \Phi)$, for each $i<\kappa$. 
It remains true for each $\eta$ that $p_\eta \in \ts(N_\eta)$ dnf over $N_{\eta, i_*}$. 
We arrive to the second use of simplicity, the independence theorem. 

\begin{subclaim} \label{sc2b}
If $\nu \in {^{i_*} \chi}$ then 
\[ q_\nu = \bigcup \{ p_\eta : \eta \in {^\kappa \chi}, ~\nu \tlf \chi \} \]
is a partial type which dnf over $N_{\trianglelefteq \nu} = \GEM_{\tau(T)}(J \rstr \{ \rho \tlf \nu \}, \Phi)$.
\end{subclaim}

\begin{proof} It suffices to consider some finite $\Lambda \subseteq {^\kappa \lambda}$ and 
prove $q_\Lambda = \bigcup \{ p_\eta : \eta \in \Lambda \}$ dnf over $N_{\tlf \nu}$.
We prove this by induction on $|\Lambda|$.  If $|\Lambda| = 1$ this is immediate since 
each $p_\eta$ is a type which dnf over $N_{\tlf \nu}$. So assume $|\Lambda| = n+1 \geq 2$. 
Let $\eta_0, \dots, \eta_{n}$ list $\Lambda$ in lexicographically increasing order. 
Let $\rho_0 = \eta_{n-1} \cap \eta_{n}$, and let $\rho = \eta_{n-1} \rstr \lgn(\rho_0) + 1$. 

Let $q_* = \bigcup \{ q_\ell : \ell \leq n-1 \}$, which by inductive hypothesis 
is a partial type which dnf over $N_{\tlf \nu}$. 
Let $q_{**}$ be a complete nonforking extension of $q_*$ to $B = \bigcup \{ N_{\tlf\eta_\ell} : \ell \leq n-1 \}$. 
That is, $q_{**} \in \ts(B)$ dnf over $N_{\tlf \nu}$, so a fortiori dnf over $N_{\tlf \rho}$. 

We have already defined $B$. For clarity, let $A = N_{\tlf \rho}$, and let $C = N_{\tlf \eta_n}$. So 
$q_n \in \ts(C)$ dnf over $N_{\tlf \nu}$, so a fortiori dnf over $A$. 

Let's first prove that $q_{**} \cup q_n$ is consistent and dnf over $A$.  
We have that $A = B \cap C$, and $q_{**} \in \ts(B)$ dnf over $A$, $q_n \in \ts(C)$ dnf over $A$, and $q_{**} \rstr A = 
q_n \rstr A$ (because they agree on any common initial segment). In order to apply the independence theorem, 
we need to know $B$ is free from $C$ over $A$.   
$A$, $B$, $C$ are universes of models of $T_0$ and by Claim \ref{n11}, $tp(C,B)$ is finitely satisfiable in $A$, 
which suffices. 

We conclude that $Q = q_{**} \cup q_n$ is a consistent partial type which dnf over $A \supseteq N_{\tlf \nu}$.  
Recalling the definition of $A$, $Q \rstr A$ is a type which dnf over $N_{\tlf \nu}$ because it is just one of the images of $p$ 
under one of the automorphisms $h$. So by transitivity of nonforking for simple theories, $Q$ dnf over $N_{\tlf \nu}$, 
and this proves the subclaim. 
\noindent\emph{Proof of Subclaim \ref{sc2}}. \end{proof}

To complete the proof of Lemma \ref{n14}, let $N_*$ be a sufficiently saturated 
elementary extension of $N$ (so, also a $\tau(\Phi)$-model) in which 
for each $\nu \in {^{i_*}\chi}$ the type $q_\nu$ is realized by some $b_\nu$.  Add to $\tau(\Phi)$ a new 
unary function symbol $F_{i_*}$. Expand $N_*$ to $N^+_*$ by interpreting $F_{i_*}$ so that 
 $\nu \in {^{i(*)}\chi}$ implies $F^{N^+_*}_{i_*}(a_\nu) = b_\nu$, where $a_\nu$ belongs to the skeleton of $N \preceq N_*$. 
In this language, note that what the subclaim has really shown is 
that for any finite sequence $\bar{\eta}$ from a single branch of $I^{\geq \nu}$ and any formula $\vp(x,\bar{a}_{\bar{\eta}})$ in the given type $p$, 
whether or not $N^+_* \models \vp[F_{i_*}(a_\nu), \bar{a}_{\bar{\eta}})]$ is a property of the quantifier-free type of $\bar{\eta}$.  
Apply the Ramsey property \ref{d:ramsey-exp}, with $J, \GEM(J, \Phi)$ and $\ma$, $N^+_*$ here for 
$J, M$ and $\ma$, $N^+$ there, to 
obtain a template $\Psi \geq \Phi$ proper for $\mk$. 
By construction, the template $\Psi$ will have registered from $f$ the correct instructions (definition) to ensure realization. 
In particular, in the model $\GEM(I, \Psi)$, for $\nu = {^{i_*}{\{ 0 \}}}$, we have that $F_{i_*}(a_\nu)$
will realize $p$. 
\noindent\emph{Proof of Lemma \ref{n14}}.  
\end{proof}

The proof of \ref{n14} remains true restricting to a set of formulas $\Delta$ which are simple.  
(Replace ``complete nonforking extension in $\ts(B)$'' by ``nonforking extension in $\ts_\Delta(B)$''.)
It's worth noting the following special case of the above argument, when $\Delta$ is finite, so we may take $\kappa = \aleph_0$.

\begin{cor}  
\label{n14a}
Let $I$ be standard with universe ${^{\omega > }{\{ 0 \}}}$. Suppose $\Delta$ is a set of formulas of $T$ 
such that every $\Delta$-type in every model of $T$ does not fork over some finite set. 
Then for every $\Phi \in \up$, there is $\Psi \in \up$ with 
$\Phi \leq \Psi$ such that $M = \GEM_{\tau(T)}(I, \Psi)$ is $\mu$-saturated for $\Delta$-types.  
\end{cor}

We now return to the main line of the present argument.

\begin{claim} \label{n17} \emph{ }
Let $I$ be standard with universe ${^{\kappa >} \{0 \}}$. For any $\Phi \in \up$, 
the model $\GEM_{\{F_*\}}(I, \Phi)$ is not $\kappa^+$-saturated.  More precisely, it omits some
partial $\varphi$-type of cardinality $\kappa$, where $\varphi = \varphi(x,\bar{y}) = (F_*(y_0, x) = y_1)$. 
\end{claim}

\begin{proof}
Let $\eta_i \in {^i\{0\}}$, and let 
\[ p({x}) = \{ F(a_{\eta_{2i}}, x) = a_{\eta_{2i+1}} : i < \kappa \} \]
be the type of a code for a function which acts as a ``successor'' operation on even 
elements in this branch of 
of the skeleton.  
Towards contradiction assume $c \in M = \GEM(I, \Phi)$ realizes $p$. So there is a $\tau(\Phi)$-term 
$\sigma(t_0, \dots, t_{n-1})$ and $i_0 < \cdots < i_{n-1} < \kappa$ such that  
\[ M \models\mbox{``} c=\sigma(a_{\eta_{i_0}}, \cdots, a_{\eta_{i_{n+1}}}).\mbox{''} \]
Let $J$ be ${^{\kappa>}\chi}$, so $J$ is standard and extends $I$.  
Let $N = \GEM(J, \Phi)$. 
Recalling that the predicates $P_k$ name elements of level $k$, let 
$\nu \in P^{J}_{2i_{n-1}+4}$ be such that 
$\nu \rstr 2i_{n-1}+2 = \eta_{2i_{n-1}+2}$, but $\nu \neq \eta_{2i_{n-1}+3}$. By the choice of $c$,
\[ N \models \mbox{``} F_*(a_{\eta_{2i_{n-1}+2}}) = a_{\eta_{2i_{n-1}+3}}  \mbox{''} \]
but then by indiscernibility we must also have  
\[ N \models \mbox{``} F_*(a_{\eta_{2i_{n-1}+2}}) = a_\nu  \mbox{''} \]
contradicting $a_{\eta_{2i_{n-1}+3}} \neq a_\nu$. 
\end{proof}

\begin{cor} \label{n22}
Let $I$ be standard with universe ${^{\kappa >} \{0 \}}$ and let $T_{1}$ be the the non-simple theory 
fixed at the beginning of the section. If $T_{1}$ has $TP_2$, then for any $\Phi \in \up$, 
$\GEM_{\tau(T_{1})}(I, \Phi)$ is not $\kappa^+$-saturated.  
\end{cor}

\begin{proof}
By our hypothesis \ref{n2} the theory represents $F_*$, so apply Claim \ref{n17}. 
\end{proof}

\begin{concl} \label{n21}
Let $I =  {^{\kappa > }{\{ 0 \}}}$. There is $\Phi \in \up$ such that writing $M = \GEM_{\tau(T)}(I, \Phi)$ 
and $N = \GEM_{\tau(F_*)}(I, \Phi)$ we have that $M$ is $\mu$-saturated but $N$ is not $\kappa^+$-saturated.  
Moreover, for this same $\Phi$, if $T_{1}$ has $TP_2$ then $\GEM_{\tau(T_{1})}(I, \Phi)$ is not $\kappa^+$-saturated. 
\end{concl}

\begin{proof} 
By Lemma \ref{n14}, Claim \ref{n17}, and Corollary \ref{n22}.  
\end{proof}

\begin{proof}[Proof of Theorem \ref{t:simple}]
There are two cases. If $T_{1}$ has $SOP_2$ then by \ref{maxl-fact} it is already maximal under $\tlf^*_1$. 
If $T_{1}$ has $TP_2$ apply Conclusion \ref{n21}. 
\end{proof}

\br

\section{Incomparability in $\tlf$ and $\tlf^*$} \label{s:incomp}

\begin{defn} For any finite $k$, let $T_{k+1, k}$ be the generic $(k+1)$-ary hypergraph which forbids a complete hypergraph on $(k+2)$-vertices; 
for $k \geq 2$, these theories were shown to be simple with trivial forking by Hrushovski. 
\end{defn}

\begin{defn}
Consider a model $M$ whose domain is partitioned by predicates $P^M$, $Q^M$. $E^M$ is an equivalence relation on $Q^M$ with infinitely many classes. 
$R^M \subseteq P^M \times Q^M$ is a binary relation. Each element of $P^M$ connects, via $R$, to precisely $n$ elements of the $n$-th equivalence class 
of $E^M$.
Define the ``canonical simple non-low theory'' $T^*$ as the theory of the existential closure of $M$. 
\end{defn} 

\begin{thm-lit} \label{thm-ap}  Assume an uncountable supercompact cardinal exists. 
For $2 \leq k < \omega$, $\mct_k$ and $\mct^*$ are incomparable in Keisler's order. 
\end{thm-lit}

Theorem A was discovered independently by D. Ulrich in 2017 \cite{ulrich}, using the equivalent simple non-low theory of Casanovas, and by the authors in 2015 \cite{MiSh:F1530}, but the latter proof was not published. 
For completeness, 
we include it in the Appendix.  

\begin{cor} \label{c52} Assuming existence of an uncountable supercompact cardinal,  
$\mct_k$ and $\mct^*$ are incomparable in $\tlf^*_1$ and $\tlf^*_{\aleph_1}$. 
\end{cor} 

\begin{thm-lit}[\cite{MiSh:1140} Theorem 7.2]  Let $k \geq 2$, let $T_{k+1,k}$ be as above, and let $T_f$ be the simple low theory from \cite{MiSh:1140} Definition 2.4. 
Then $T_{k+1,k}$ and $T_f$ are incomparable in Keisler's order, in ZFC. 
\end{thm-lit}

\begin{cor}
Let $k \geq 2$ and let $T_{k+1,k}$ and $T_f$ be as above. Then  $T_{k+1,k}$ and $T_f$ are incomparable in $\tlf^*_1$ and 
$\tlf^*_{\aleph_1}$, in ZFC. 
\end{cor}

Theorem A remains valuable after the discovery of Theorem B since the theories are different. In some sense, the mechanism of Theorem A works by leveraging forking against independence,  as explained below, whereas the mechanism of Theorem B works on two low theories with trivial forking, 
leveraging imperfections coming from trees against imperfections coming from amalgamation. 

\begin{qst}  \label{q:one}
Do there exist proofs of incomparability in $\tlf^*_{\aleph_1}$ having no analogue in Keisler's order? For example, can \ref{c52} be proved 
directly in ZFC? 
\end{qst}

\begin{disc} \label{d:two}
\emph{A model-theoretic value of the proof of Theorem A lies in highlighting a certain independence between complexity in the sense of forking and complexity in the sense of independence in these simple theories. On the set-theoretic side, it gives yet another proof of separation of OK and good, under large cardinal hypotheses, addressing a question of Dow from 1985 (which we had previously proved under the assumption of a measurable 
cardinal \cite{MiSh:996}): in this sense, we may think of Theorem A as allowing for a further calibration of the level of goodness of various OK ultrafilters, 
using the $T_{k+1,k}$'s, e.g.: }
\end{disc} 

\begin{cor} \label{c:ok} Assuming $k \geq 2$, $\lambda, \mu, \theta, \sigma$ are suitable, $\lambda = \mu^{+k+1}$,  and $\theta = \sigma$ is uncountable and supercompact, there exists a regular ultrafilter on $\lambda$ which is OK, and good for the random graph, but not good for $T_{k+1,k}$. 
\end{cor}

Our earlier separation of good and OK using a measurable produced an ultrafilter not good for the random graph; here the 
cardinal assumption is stronger, but also the ultrafilter is closer to being good.

\section{Discussion: weak definability of types} \label{s:wd}

From our proofs of Theorems \ref{t:trg-s} and \ref{t:simple} one may extract the following principle.

\begin{hyp} \label{d:cont1}
Fix for this section: 
\begin{enumerate}[a)]
\item a theory $T$.
\item a class of index models $\mk$ satisfying the Ramsey property $\ref{d:ramsey-exp}$.
\item a class $\Upsilon$ of templates $\Phi$ proper for $\mk$ satisfying $\ref{c:nice}$, and with $\tau(\Phi) \supseteq T$ for each $\Phi \in \Upsilon$, 
recalling that $\ref{c:nice}$ implies $T_\Phi$ is well defined and has Skolem functions.
\item $\leq$ the natural order on $\Upsilon$. 
\item \emph{thus} the set $\operatorname{Terms}$ of $\tau(\Phi)$-terms.
\end{enumerate}
\end{hyp}

\begin{defn} \label{d:gd}
Suppose $I \in \mk$, $\Phi \in \Upsilon$, $M = \GEM(I, \Phi)$ with skeleton $\ma$, $\Delta$ is a set of $\ml(\tau_T)$-formulas, $p$ a partial type 
$p \subseteq q \in \ts_\Delta(M)$. 
We may say $p$ has a \emph{weak  definition} 
if there is a partial function 
\[ F: \Delta \times {^{\omega>}(\operatorname{Terms})} \times D_{\operatorname{qf}}(I) \rightarrow \{ 0, 1 \} \] such that 
for some $\aleph_0$-saturated $J \in \mk$,  
when evaluated in $N = \GEM_{\tau(T)}(J, \Phi)$, the set of formulas
\begin{align} \label{gd:eqn}
\begin{split} 
\{ \vp(\bar{x}, \bar{\sigma}(\bar{a}_{\bar{t}}))^\trv & \colon \hspace{5mm} \vp \in \Delta, ~ \bar{\sigma} \in {^{\omega>}(\operatorname{Terms})}, \\  
& \hspace{8mm} \bar{t} \subseteq J,  ~\qftp(\bar{t}, \emptyset, J) = \xr \\ 
& \hspace{4mm} \mbox{  and }\trv = F(\vp, \bar{\sigma}, \xr) \in \{ 0, 1 \} \} \\
\end{split}
\end{align}
 is a partial type which 
 extends $p$. 
\end{defn}

Note that if $\bar{t}$, 
$\bar{\sigma}$ don't have the appropriate length or size for the given $\vp$, the function $F$ from may be undefined on the 
tuple $(\vp, \bar{\sigma}, \xr)$; but order to meet the condition that (\ref{gd:eqn}) extends $p$, 
$F$ will need to be defined on all of the tuples $(\vp, \bar{\sigma}, \xr)$ which arise from $p$.\footnote{Moreover, if $\vp(x_0, \dots, x_{m-1}, y_0, \dots, y_{n-1}) \in \Delta$, and $\bar{\sigma} = \langle \sigma_0, \dots, \sigma_{n-1} \rangle$ 
is a finite sequence from $\operatorname{Terms}$, then without loss of generality (by adding dummy variables) we may assume these terms all have 
the same number $r$ of free variables, and so if $\bar{a}_{\bar{t}}$ is from the skeleton and $\ell(\bar{a}_{\bar{t}})=r$, 
we may write ``$\vp(\bar{x}, \bar{\sigma}(\bar{a}_{\bar{t}}))$'' 
for $\vp(\bar{x}, \sigma_0(\bar{a}_{\bar{t}}), \dots, \sigma_{n-1}(\bar{a}_{\bar{t}}))$.}  So this condition 
does generalize e.g. $(\ref{eq:q})$ from the proof of Claim \ref{m17}. 

\begin{rmk} \label{rmk15}
Definition \ref{d:gd} can be extended to include weak definitions over some finite $\bar{t}^* \subseteq I$, but since this was not 
used in the present proofs, we defer this to the companion paper \cite{MiSh:F1692}.  
\end{rmk}

\begin{claim} \label{o:stabledef} 
Suppose $T_\Phi$ has Skolem functions for $T$.  
If $p$ has a definition over the empty set in $M$, 
 a finite subset of $M$ in the usual sense of stability theory, then $p$ has a weak  definition in the 
sense of Definition $\ref{d:gd}$.
\end{claim}

\begin{proof}[Proof of \ref{o:stabledef}] 
If $p$ is definable over $\emptyset$, then for each $\vp(\bar{x}, \bar{y}) \in \Delta$ there is a $\tau(T)$-formula $d_\vp(\bar{y})$ giving the definition.   
Fix $\bar{\sigma} \in {^{\ell(\bar{y})}(\operatorname{Terms})}$ and consider any finite sequence $\bar{t} \in {^{\omega>}I}$ for 
which $\bar{\sigma}(\bar{a}_{\bar{t}})$ can be evaluated.  Let 
$\xr = \qftp(\bar{t}, \emptyset, J)$.  Since $\Phi$ is a template, for all $\bar{s}$ with 
$\qftp(\bar{s}, \emptyset, J) = \xr$, 
\[ \qftp(\bar{a}_{\bar{t}}, \emptyset, N) = \qftp(\bar{a}_{\bar{s}}, \emptyset, N).\]
The assumption that $\tau(\Phi)$ has Skolem functions for $T$ improves this to 
\[ \tp(\bar{a}_{\bar{t}}, \emptyset, N) = \tp(\bar{a}_{\bar{s}}, \emptyset, N). \]
In particular, 
\[ N \models d_\vp(\bar{\sigma}(\bar{a}_{\bar{t}})) \iff N \models d_\vp(\bar{\sigma}(\bar{a}_{\bar{s}})). \]
Note Skolem functions are not needed for quantifier-free definitions. 
\end{proof}

\begin{obs} 
If $p$ has a weak  definition in the 
sense of Definition $\ref{d:gd}$, this does \emph{not} imply $p$ has a definition in the usual sense of stability theory, 
even assuming Skolem functions for $T$. 
\end{obs}

\begin{proof}
 Existence of definitions is characteristic of stability; earlier sections built weak definitions for types in the random graph and in 
arbitrary simple theories, respectively. 
\end{proof}

We may summarize by noting that in each case, the contribution of weak definability was to prove a lemma of the following kind.  
 
\begin{mlem} \label{mlem1}
Suppose $I \in \mk$, $M = \GEM_{\tau(T)}(I, \Phi)$ and 
\[ p(x) = \{ \vp_\alpha(\bar{x}, \bar{\sigma}^M_\alpha(\bar{a}_{\bar{t}_\alpha})) : \alpha < \kappa \} \]
is a consistent partial type in $M$. If $p$ has a weak definition, then for any $\aleph_0$-saturated 
$J$ with $I \subseteq J \in \mk$, 
the set of formulas 
\begin{align*}
\label{eq:qc} q(x) = \{ \vp_\alpha(\bar{x}, \bar{\sigma}^M_\alpha(\bar{a}_{\bar{s}})) : \alpha < \kappa, ~
\tpqf(\bar{s}, \emptyset, J) = 
\tpqf(\bar{t}_\alpha, \emptyset, I) \} 
\end{align*}
is a consistent partial type in $N$. So we may realize it in some elementary extension $N^\prime$ of $N$ 
and name this realization by new constants $\bar{c}$, and applying $\ref{d:ramsey-exp}$, 
we may find $\Psi \geq \Phi$ such that $\bar{c} \subseteq \tau(\Psi)$ 
and $p$ is realized by $c$ in $N = \GEM(I, \Psi)$. 
\end{mlem}

\begin{disc} \emph{Our proofs have suggested that a productive way of comparing theories may be to find, 
in the setup of $\GEM$-models, a class $\mk$
for which types in one theory $T_0$ have weak definitions, and those in another $T_1$ do not.}
\end{disc}

\section{Some open problems} \label{s:op}

We conclude with some open problems. 
The careful reader may also have noticed many natural questions which we have not addressed here, 
for example 
to extend Lemma \ref{l:trg} (and the analogous proof for $\tlf$) to show that all stable theories are $\tlf^*$-below all unstable theories.

Note that no equivalence classes of unstable theories under $\tlf^*$ have been characterized in ZFC (though maximality 
uses only instances of GCH) 
and any result along these lines could potentially be very interesting. 

\br
Towards understanding $\tlf^*$ on the simple unstable theories, for $\kappa = 1$ or $\kappa = \aleph_1$:
\begin{enumerate}
\item Characterize those theories which are $\tlf^*_\kappa$-equivalent to the theory of the random graph. 

\item Are there infinitely many incomparable classes of simple unstable theories under $\tlf^*_\kappa$? 

\item Is it true that every simple theory is $\tlf^*_\kappa$-below every non-simple theory?

\end{enumerate}

Towards understanding $\tlf^*$ on the non-simple theories with $NSOP_2$: 

\begin{enumerate}

\item[(4)] \label{p:nsop2} Prove Fact $\ref{844-fact}$ in ZFC, which would establish in ZFC that a theory is maximal in $\tlf^*_{\aleph_1}$ if and only if it is $SOP_2$. 

\item[(5)] Characterize those theories which are $\tlf^*_\kappa$-equivalent to $\tfeq$.  

\item[(6)] Is there a property of non-simple theories, which is analogous in a natural sense 
to f.c.p. in stable theories and to non-lowness in simple unstable theories, and is detected as a division in $\tlf^*_\kappa$? 

\end{enumerate}

\section*{Appendix: On incomparability}  \label{appendix-b}

\noindent Continuing \S \ref{s:incomp}, we prove Theorem A following our earlier unpublished proof. 
We encourage the reader interested in incomparability to also read Ulrich's proof \cite{ulrich}; the core 
mechanism is the same, but one learns different things from different people.

We will use the following key ingredients of the proof of infinitely many classes from \cite{MiSh:1050}.  
The reader may take the properties of ``optimized'' and ``perfected'' to be black boxes. Such ultrafilters were defined 
and proved to exist in \cite{MiSh:1030} \S 5 and \S 9, respectively.\footnote{for so-called ``suitable'' cardinals $\lambda \geq \mu \geq \theta \geq \sigma$: 
defined in \cite{MiSh:1030} Definition 1.1.} Note that existence proof for optimized ultrafilters and 
uncountable $\sigma$ from \cite{MiSh:1030} assumed $\sigma$ supercompact. This is probably more that is needed, however, 
the existence result uses that a certain ultrafilter $\de_*$ is $\sigma$-complete.

\begin{thm-lit}[\cite{MiSh:1050}
\footnote{Theorem 6.1. The non-saturation result depends only on $\lambda$ and $\mu$, see Remark 5.2 there.} 
] \label{t:p2a}
Let $(\lambda, \mu, \theta, \sigma)$ be suitable. 
Suppose $\mu = \aleph_\alpha$ and $\lambda = \aleph_{\alpha + \ell}$ for $\alpha$ an ordinal and $\ell$ a nonzero integer. 
Suppose that either:
\begin{enumerate}
\item[(i)] $\theta = \sigma = \aleph_0$ and $\de$ is a $(\lambda, \mu)$-perfected ultrafilter on $\lambda$.
\item[(ii)] $\theta \geq \sigma > \aleph_0$, so $\sigma$ is supercompact, 
\\ and $\de$ is a $(\lambda, \mu, \theta, \sigma)$-optimized ultrafilter on $\lambda$.
\end{enumerate}
Then for any $2 \leq k < \omega$: 
\begin{enumerate}
\item[(a)] If $k < \ell$,  $\de$-ultrapowers of models of $T_{k+1, k}$ are not $\lambda^+$-saturated. 
\item[(b)] If $\ell < k$ and (i), $\de$-ultrapowers of models of $T_{k+1, k}$ are $\lambda^+$-saturated.\footnote{Our 2015 manuscript 
states (b) in full generality:  if $\ell < k$, $\de$-ultrapowers of models of $T_{k+1, k}$ are $\lambda^+$-saturated. This requires some thought  
beyond \cite{MiSh:1050} Theorem 6.1 on the reader's part, e.g. verifying that \cite{MiSh:1050} gives explicit simplicity for $T_{k+1,k}$ and then applying \cite{MiSh:1030} Theorem 7.3.}

\end{enumerate} 
\end{thm-lit}

\begin{claim} \label{sat-low}
If $\sigma > \aleph_0$ $($so $\de_*$ is $\sigma$-complete$)$, $\de$ is flexible and is good for the random graph\footnote{Flexible by  
\cite{MiSh:1030} 5.16, good for the random graph by \cite{MiSh:1030} Theorem 7.3 as the random graph trivially satisfies 
the condition of being explicitly simple, see e.g. \cite{MiSh:1030} Discussion 3.14.} and 
$\de$ is good for $T_*$. 
\end{claim}

\begin{proof}
Let $M \models T_*$, $p \in \ts(N)$ where $N \preceq M^I/\de$ and $||N|| \leq \lambda$. 

Let $\langle \vp(x,a^*_\alpha) : \alpha < \lambda \rangle$ be an enumeration of $p$. 
In the main case we may assume each $\vp$ is of the form 
$R(x,a^*_\alpha)^{\trv(\alpha)}$.  
As $\de$ is good for the random graph, it will suffice to consider the case where each $\trv(\alpha) = 1$. 
Let $M_*$ be a countable model over which $p$ does not fork.  So $M_*$ contains the prime model, and $p \rstr M_*$ 
includes the data of which $n$ elements $p$ connects to in the $n$th class of $E$, for each finite $n$. Denote these
by $\{ a_{\langle i,n \rangle} : i < n, n < \omega \}$, where $\langle ~ \rangle$ is some fixed coding function from $\omega \times \omega$ 
to $\omega$, and assume that each $a_{\langle i, n \rangle}$ is $a^*_\alpha$ for $\alpha = \langle i, n \rangle < \omega$. 
Let $\langle \mb_u : u \in [\lambda]^{<\sigma} \rangle$ 
be the continuous sequence given by 
\[ \mb_u = \bigcap_{\alpha \in u, i < n < \omega, \bigwedge_{j < n}\alpha \neq \langle j,n \rangle}  \ma_{\neg(E(x_\alpha, x_{\langle i,n \rangle}))}. \] 
With this sequence we may realize $p$. 
\end{proof}

\begin{defn} 
Let $\mct_n$ be the theory given by the disjoint union of $T_{k+1,k}$ for $k \geq n$. 
Let $\mct^*_n$ be the theory given by the disjoint union of $\mct_n$ and $T_*$.  
\end{defn}

\begin{claim} \label{c:upgrade}
Let $(\lambda, \mu, \theta, \sigma)$ be suitable. 
Suppose $\mu = \aleph_\alpha$ and $\lambda = \aleph_{\alpha + \ell}$ for $\alpha$ an ordinal and $\ell$ a nonzero integer. 
Suppose that $\theta \geq \sigma > \aleph_0$, so $\sigma$ is supercompact, 
\\ and $\de$ is a $(\lambda, \mu, \theta, \sigma)$-optimized ultrafilter on $\lambda$.
Then for any $2 \leq k < \omega$: 
\begin{enumerate}
\item[(a)] If $k < \ell$, then $\de$-ultrapowers of models of $T_{k+1, k}$ are not $\lambda^+$-saturated, \emph{moreover} 
the same is true for $\mct^*_k$. 

\item[(b)] If $\ell < k$, then $\de$-ultrapowers of models of $\mct_*$ are $\lambda^+$-saturated.
\footnote{Continuing 
footnote 18, our manuscript 
had the more general statement: $\ell < k$, then $\de$-ultrapowers of models of $\mct_*$ are $\lambda^+$-saturated, 
\emph{moreover} the same is true for $\mct^*_k$.  Likewise, $\mct_*$ was the more general $\mct^*_n$ in the main incomparability result. 
Notice this gives some interesting additional information, namely, that assuming existence of a supercompact cardinal, 
there is an infinite descending chain of non-low simple theories in Keisler's order.}
\end{enumerate} 
\end{claim}

\begin{proof}
By Claim \ref{sat-low} and Theorem \ref{t:p2a}.  
\end{proof}

\begin{proof}[Proof of Theorem A]
If $\sigma = \theta = \aleph_0$, $\lambda \leq \mu^{+k}$ and $\de$ is $(\lambda, \mu)$-perfected, then $\de$ is good for 
$\mct_k$. However, $\de$ is not flexible,\footnote{by \cite{MiSh:999} Corollary 9.9.} so it is not good for $\mct_*$. 
If $\sigma = \theta$ is an uncountable supercompact cardinal, but $\lambda = \mu^{+k+1}$, then $\de$ is good for 
$\mct_*$ 
by \ref{c:upgrade}(b), but it is not good for $\mct_{k^\prime}$ for any $k^\prime \leq k$ by \ref{c:upgrade}(a). 
\end{proof}

\end{document}